\documentclass[11pt]{amsart}

\usepackage[utf8]{inputenc}
\usepackage[T1]{fontenc}
\usepackage[french, english]{babel}

\usepackage{amsthm}
\usepackage{amsmath}
\usepackage{amssymb}
\usepackage{mathrsfs}
\usepackage{bbm}
\usepackage{graphicx}
\usepackage{listings}
\usepackage{mathrsfs}
\usepackage{enumerate}
\usepackage{pict2e}
\usepackage{stmaryrd}
\usepackage{mathtools}

\usepackage[all,cmtip]{xy}

\usepackage{multirow}

\usepackage{color}
\definecolor{marin}{rgb}   {0.,   0.1,   0.5} 
\definecolor{rouge}{rgb}   {0.8,   0.,   0.} 
\definecolor{sepia}{rgb}   {0.4,   0.25,   0.} 
\definecolor{mag}{rgb}   {0.3,   0,   0.3} 

\usepackage[colorlinks,citecolor=sepia,linkcolor=marin,urlcolor=mag ,bookmarksopen,bookmarksnumbered]{hyperref}

\newtheorem{theorem}{Theorem}[section]
\newtheorem{corollary}[theorem]{Corollary}
\newtheorem{lemma}[theorem]{Lemma}
\newtheorem{proposition}[theorem]{Proposition}
\newtheorem{definition}[theorem]{Definition}
\newtheorem{remark}[theorem]{Remark}

\newcommand{\eps}{\varepsilon}

\begin{document}

\title[A  result of the type Nekhoroshev  for  NLS on $\mathbb T^d$ in low regularity
]{Almost global existence for some nonlinear Schr\"odinger equations on $\mathbb{T}^d$ in low regularity}

\author{Joackim Bernier}

\address{\small{Nantes Universit\'e, CNRS, Laboratoire de Math\'ematiques Jean Leray, LMJL,
F-44000 Nantes, France
}}

\email{joackim.bernier@univ-nantes.fr}

\author{Beno\^it Gr\'ebert}

\address{\small{Nantes Universit\'e, CNRS, Laboratoire de Math\'ematiques Jean Leray, LMJL,
F-44000 Nantes, France
}}

\email{benoit.grebert@univ-nantes.fr}



\begin{abstract} 
We are interested in the long time behavior of solutions of the nonlinear Schr\"odinger equation on the $d$-dimensional torus in low regularity, i.e. for small initial data in the Sobolev space $H^{s_0}(\mathbb T^d)$ with $s_0>d/2$. We prove that, even in this context of low regularity, the $H^s$-norms, $s\geq 0$, remain under control during times,  $T_\varepsilon= \exp \big(-\frac{|\log\varepsilon|^2}{4\log|\log\varepsilon|} \big)$, exponential with respect to the initial size of the initial datum in $H^{s_0}$, $\|u(0)\|_{H^{s_0}}=\varepsilon$. For this, we add to the linear part of the equation a random Fourier multiplier in $\ell^\infty(\mathbb Z^d)$ and show our stability result for almost any realization of this multiplier. In particular, with such Fourier multipliers, we obtain the almost global well posedness of the nonlinear Schr\"odinger equation on $H^{s_0}(\mathbb T^d)$ for any $s_0>d/2$ and any $d\geq1$.

\end{abstract} 
\maketitle

\setcounter{tocdepth}{1} 
\tableofcontents

\section{Introduction}
The long time behavior of solutions of Hamiltonian partial differential equations has been a major issue in the PDE community for two decades. A central question, initially posed by Bourgain \cite{Bou96}, concerns the possibility that a solution sees its Sobolev $H^s$-norm  tend to infinity when time tends to infinity for $s$ large enough although the energy - the Hamiltonian - is conserved (see also \cite{Bou04a, Bou04b}). Such a behavior would clearly contrast with the behavior of solutions of linear PDEs. A number of results have been obtained to postpone this eventuality to very long times with respect to the size of the initial data (see e.g. \cite{Bou96,BG06,BDGS07,BD17,KillBill,BMP20,FI19}). Unfortunately they only concern very regular solutions (the larger is $s$, the longer is the stability time) contrary to what the numerical simulations suggest (\cite{CHL08a, CHL08b}). On the contrary we consider here low regularity solutions of an emblematic Hamiltonian PDE, namely the nonlinear Schr\"odinger equation. Concretely we consider the Cauchy problem 
\begin{equation}\label{NLS}\tag{NLS}
\left\{
\begin{array}{lll}
i\partial_t u&= -\Delta u +V*u +\sigma |u|^{2p}u,\\
u(0)&=u^{(0)},
\end{array} \right.
\end{equation}
where $t\in \mathbb{R}$, $x\in \mathbb{T}^d = (\mathbb{R}/2\pi \mathbb{Z})^d$,  $d\geq 1$, $p\geq 1$ is an integer, $\sigma \in \{-1,1\}$ allows to consider both the focusing and the defocusing cases and $u\mapsto V* u$ is a Fourier multiplier with a potential  $V\equiv (V_k)_{k\in\mathbb Z^d}\in\ell^\infty(\mathbb Z^d;\mathbb{R})$ whose Fourier coefficients are real and bounded. More precisely, we identify every function $v \in L^2(\mathbb{T}^d)$ with the sequence of its Fourier coefficients
$$
v_k  = (2\pi)^{-d/2} \int_{\mathbb{T}^d} v(x) e^{-i k \cdot x} \mathrm{d}x, \quad k\in \mathbb{Z}^d
$$ and so $V* u$ is defined by the relation $(V* u)_k=(2\pi)^{-d/2} V_k u_k$. Our main result is
\begin{theorem}\label{th-main}
There exists a non empty set $\mathcal V\subset \ell^\infty(\mathbb Z^d)$ such that for $V\in\mathcal V$ and $s_0>d/2$, there exits $\eps_0\equiv \eps_0(s_0,V,d,p) >0$ such that for any  $u^{(0)}\in H^{s_0}(\mathbb T^d)$ satisfying $\varepsilon := \|u^{(0)}\|_{H^{s_0}} \leq \varepsilon_0$,
 the Cauchy problem \eqref{NLS} has a unique solution 
 \begin{center}
 $
 u\in C^0((-T_\eps,T_\eps);H^{s_0}(\mathbb T^d))\cap  C^1((-T_\eps,T_\eps);H^{s_0-2}(\mathbb T^d))\ $
with 
$
\ T_\eps= \eps^{-\frac{|\log\eps|}{4\log|\log\eps|}}.
$
\end{center}
 If moreover $u^{(0)}\in H^{s}(\mathbb T^d)$ for some $s\geq 0$ then $u\in C^0((-T_\eps,T_\eps);H^{s}(\mathbb T^d))$ and
\begin{equation}
\label{nous} 
\| u(t)\|_{H^s} \leq C_s \|u^{(0)}\|_{H^s} \quad \mathrm{for } \quad |t|\leq T_\eps
\end{equation}
where $C_s\geq 1$ is a constant depending only on $s$.
\end{theorem}

\begin{remark} The estimate \eqref{nous} means that, to control  the growth of the $H^s$ norm for very long times, the initial datum only needs to be small in $H^{s_0}$. In particular, as recently highlighted in \cite{FM22}, contrary to what is usually assumed, it does not have to be small in $H^s$ (moreover note that we do not have to assume that $s\geq s_0$).
\end{remark}

\subsection{Context and further bibliographical comments} Let us first situate this theorem in regard to previous results. First, we point out that the local well-posedness of \eqref{NLS} in $H^s$, $s>d/2$, provides a similar result but for much shorter times: we would only have $T_\varepsilon \simeq \varepsilon^{-2p}$. On the side of stability over long times, the work of Bambusi and Bambusi--Grebert (see \cite{Bam03, BG03, BG06}) established, by a normal form method, the following result concerning \eqref{NLS}: for large families of Fourier multipliers $V \in H^m(\mathbb{T}^d)$ with $m>0$, for $r\gg 1$ chosen arbitrarily large, $s\geq s_0(r)\gtrsim r^2$ and $\|u^{(0)}\|_{H^s}=\eps$ small enough, the existence time  of the solution of \eqref{NLS} is larger than $\varepsilon^{-r}$ and we have
$$ 
\| u(t)\|_{H^s}\leq 2 \eps \quad \text{for } |t|\leq  \eps^{-r}.
$$
In this result the time of stability is directly related to the regularity of the solution, we have stability in $H^s$ for time of order $\eps^{-c\sqrt s}$ where $c>0$ is a universal constant. In \cite{BMP20}, in dimension $d=1$, this time has been enlarged to  $\eps^{- cs}$. We also note that this result was extended in \cite{KillBill} to the case with $V=0$ but then for random initial data. The main flaw in all these results is that they apply only in very high regularity although the corresponding partial differential equations are locally well-posed in low regularity. Theorem \ref{th-main} relaxes this constraint: $s$ only need to be larger than $d/2$, this last constraint coming from the fact that we want to work in an algebra. We also note that in the result of Bambusi--Grebert, the coefficients of the Fourier multipliers are decreasing and actually correspond to potentials in $H^m$ which is not the case in Theorem \ref{th-main}. Moreover, the potential, $V\in \mathcal V$, we actually consider are not very generic in $\ell^\infty$ : the eigenvalues of the operator $u \mapsto V \ast u$ have a lot of multiplicities (see \eqref{eq:def_V_rand}). As we will see, this specificity is a key ingredient allowing us to have the stability result (see subsections \ref{sec:ideaproof} and \ref{sec:NR} for details).

On the side of instability, Colliander--Keel--Staffilani--Takaoka--Tao (see~\cite{CKSTT}),  considered  the cubic  nonlinear Schr\"odinger equation, on the two dimensional torus $\mathbb T^{2}$ without Fourier multipliers ($V=0$ in \eqref{NLS})
and proved that  for any $\varepsilon\ll 1$, any $K \gg 1$ and $s>1$ there exists a solution $u$ and a time $T $ such that 
$$
\|u(T )\|_{H^s} \geq K \quad \mathrm{and} \quad \|u(0 )\|_{H^s} \leq \varepsilon.
$$   
After that, Guardia--Kaloshin (see~\cite{GK15,GKerra}), proved a quantitative estimate on $T$ :
$$
0< T \leq e^{ \big( \frac{K}{\varepsilon} \big)^c}
$$
where $c>0$ is a constant depending only on $s$. A maybe less intuitive extension is then obtained by  M. Guardia (see~\cite{Gua14}): he proves that this "almost unbounded" behavior is not a consequence of the exact resonances, since it persists when one adds a convolution potential $V$, i.e. for \eqref{NLS} with $V\in H^{m}(\mathbb{T}^2) $, $m>0$.
Of course the stability time we obtain, although of exponential type,  is much shorter than the instability time obtained by Guardia. Furthermore  \eqref{nous} is obtained for bounded Fourier multipliers while the one of \cite{Gua14} requires a polynomial decay of these multipliers. In other words, to obtain the instability in \cite{Gua14} the linear frequencies are supposed to be asymptotically close to the fully resonant situation (i.e. $V=0$) while for \eqref{nous} this is not the case. We conjecture that our result is still true with some slow polynomial decay\footnote{i.e. $V$ would belong to some low regularity Sobolev spaces.} on the Fourier multipliers but we can glimpse a proof only for $d=1$. Nevertheless it appears that the fact of being asymptotically close (in Fourier variables) to the resonant case is decisive for the appearance of weak turbulence phenomena (of course we could easily extend our result to the case where $(V_k)_{k\in\mathbb Z^d}$ decreases logarithmically but this is not fair since this does not put $V$ in a Sobolev space).  \\
Regarding to the last sentence, it deserves to mention the following less turbulent case: Hani--Pausader--Tzvetkov--Visciglia considered in~\cite{HPTV} the cubic nonlinear Schr\"odinger equation on the wave-guide manifolds $\mathbb R\times \mathbb T^d$
 \begin{equation} \label{NLST2R}
 i\partial_t u +\Delta_{\mathbb{R}\times\mathbb{T}^d} u=\vert u\vert^2u , \quad (t,x,y)\in \mathbb R\times \mathbb R\times \mathbb T^d,
\end{equation}
and proved  that  when  $2\leq d\leq 4$ the equation admits unbounded solutions in $H^s$ for $s$ large enough.  However  when adding a "typical" convolution potential $V$ to~\eqref{NLST2R},
 it is proved in \cite{GPT} that all the small solutions remain bounded in $H^s$. So in this less turbulent case, the fact of having an exactly resonant linear part is decisive for the appearance of weak turbulence phenomena.\\
 We also mention a recent result by Giulliani--Guardia where the authors proved that the Colliander--Keel--Staffilani--Takaoka--Tao ideas still apply when considering irrational tori (see \cite{GG22}).

In finite dimension $n$, the standard Nekhoroshev result \cite{Nek77} controls the dynamics over times of order $\exp \big( -\alpha \eps^{-1/(\tau +1)} \big)$ for some $\alpha>0$ and $\tau>n+1$   (see for instance \cite{ben85, GG85, P93}) which is, of course, much better than $T_\varepsilon = \exp \big(-\frac{|\log\eps|^2}{4\log|\log\eps|} \big)$ Nevertheless, clearly this standard result does not extend to the infinite dimensional context, i.e. when $n\to +\infty$. Actually this kind of exponential times $ \exp \big(-\alpha \frac{|\log\eps|^2}{\log|\log\eps|} \big)$ were obtained by Benettin-Fr\"ohlich-Giorgilli in \cite{BFG88} for a Hamiltonian system with infinitely many degrees of freedom but with finite-range couplings. We also notice that this time was suggested by Bourgain as the optimal time that we could obtain in an analytical context (see eq. (2.14) in   \cite{Bou04b}).
 We note that in \cite{FG13} a Nekhoroshev result for \eqref{NLS} equation was proved in an analytical context and for time of order $\eps^{-\alpha|\ln \eps|^\beta}$ with $\beta<1$. See also mention \cite{BMP20} for results in Gevrey regularity and \cite{CCMW22} for a result in class of regularity between $C^\infty$ and Gevrey. We note that all these results with exponential stability time were proved for very regular solutions while we only assume, in Theorem \ref{th-main}, $H^{s_0}$ regularity with $s_0>d/2$. Actually, in our case, the exponential time is linked to the very good control of the small divisors that we obtain for Fourier multipliers in  $\mathcal V$ (see \eqref{eq:new_SD_est}).

\medskip

We point out that our theorem implies the so called almost global well-posedness of \eqref{NLS} on $H^s(\mathbb T^d)$ for any $d\geq 1$ and any $s>d/2$ when the Fourier multiplier $V$ is chosen in the non empty set $\mathcal V$. 

\medskip

We also point out that we have recently shown (see \cite{BG21}) a normal form result for \eqref{NLS} in weak regularity (in fact in the energy space $H^1$) in dimension $d=1$ or $d=2$ and almost surely with respect to the random Fourier coefficients of  potentials $V\in H^m$, $m>0$.  On one hand this later result is better because it holds true for almost all Fourier multipliers but the price we pay is that we essentially control only a finite (but large!) number of Fourier modes and in particular it cannot prove the almost global well-posedness. We note that both results are based on a new way of estimating the so-called small divisors. This result has been extended to the nonlinear Klein--Gordon in $\mathbb{S}^2$ in \cite{BGR}. As in this paper but for other reasons, it  was crucial to use partially resonant Birkhoff normal form, precisely normal forms that decompose the dynamics on large blocks of modes with an increasing size of the blocks (see \eqref{eq:def_V_rand} in our case). 

\bigskip

\subsection{Ideas of the proof} \label{sec:ideaproof}
We now give an idea of the proof of Theorem \ref{th-main}. We do not try to explain the Birkhoff normal form procedure (which is quite classical, an introduction can be found in \cite{Bam07} or in \cite{G}) but rather to present the novelties of this paper.

First let us note that to control the $H^s$ of the solution it is enough to control the observable
\begin{equation}
\label{eq:def_Ns_J_n}
\mathcal N_s(u)=\sum_{n\geq0}(2^{n})^{2s} J_n \quad\mathrm{where} \quad J_n = \sum_{k\in B_n}|u_k|^2
\end{equation}
and $\mathbb Z^d=\bigcup_{n\geq 0} B_n$ stands for the standard dyadic decomposition\footnote{We point out that we chose this standard dyadic decomposition just for simplicity. Our result could be easily extended provided that the size of the blocks grows exponential fast.} of the Fourier space
\begin{equation}
\label{eq:def_dya}
\forall n \geq 1, \quad B_n=\{k\in\mathbb Z^d \ | \  2^n\leq |k|<2^{n+1}\} \quad\mathrm{and} \quad B_0=\{k\in\mathbb Z^d \ |\ |k|<2\}.
\end{equation}
Indeed, it is clear that $\sqrt{N_s}$ is a norm which is equivalent to the standard $H^s$ norm
$$ 
2^{-2s}\|u\|_{H^s}^2\leq \mathcal N_s(u)\leq \|u\|_{H^s}^2.
$$
Therefore, to control the variations of the $H^s$ norm of the solutions it is enough to control the (relative) variation of the \emph{super-actions} $J_n$. Note that it is useless to control the variations of each action $|u_k|^2$ or of the standard super-actions $\sum_{|k|=m} |u_k|^2$ (as it is usually done, see e.g. \cite{BG06,FG13,FGL13,YZ14,KillBill,BMP20,FGI20,FI19}). 

Concretely, it means that in the Birkhoff normal form procedure, we only have to remove all the monomials of the form $u_{\boldsymbol{k}_1} \dots u_{\boldsymbol{k}_q} \overline{u_{\boldsymbol{\ell}_1}} \dots \overline{u_{\boldsymbol{\ell}_q}}$, where $2\leq q$ and $\boldsymbol{k},\boldsymbol{\ell} \in (\mathbb{Z}^d)^q$ are such that 
\begin{equation}
\label{eq:monomials_to_remove}
\exists n\in \mathbb{N}, \quad \sharp \{ j\in \llbracket 1,q \rrbracket \ | \ \boldsymbol{k}_j \in B_n \} \neq \sharp \{ j\in \llbracket 1,q \rrbracket \ | \ \boldsymbol{\ell}_j \in B_n \}.
\end{equation}
Indeed, it is simple to check that the remaining ones commute with the super actions $J_n$. Therefore it is enough to control the \emph{small divisors} 
$$
\Omega(\boldsymbol{k},\boldsymbol{\ell})  :=2i( \omega_{\boldsymbol{k}_1} + \cdots +  \omega_{\boldsymbol{k}_q} -  \omega_{\boldsymbol{\ell}_1} - \cdots -  \omega_{\boldsymbol{\ell}_q})
$$
whenever $(\boldsymbol{k},\boldsymbol{\ell})$ is \emph{non-resonant} (i.e. it satisfies \eqref{eq:monomials_to_remove}) and the frequencies\footnote{which are the eigenvalues of the linearized vector field.} $\omega_k$, $k\in \mathbb{Z}^d$,  are defined by
$$
\omega_k := |k|^2 + (2\pi)^{-d/2} V_k.
$$
For potentials $V \in H^{m}$, $m>0$, drawn following classical probability laws, it is standard to establish lower bounds of the kind (see e.g. \cite{BG06})
\begin{equation}
\label{eq:standard_SD_est}
|\Omega(\boldsymbol{k},\boldsymbol{\ell})| \geq \gamma(q,V)\, \big(\max_{1\leq j\leq q} (\langle \boldsymbol{k}_j \rangle, \langle \boldsymbol{\ell}_j \rangle)\big)^{-\alpha q} \quad \mathrm{whenever} \quad \boldsymbol{k} \neq \boldsymbol{\ell} \quad \mathrm{up \ to \ a \ permutation}
\end{equation}
where $\alpha>0$ is a constant depending only on $m$ and $\gamma(q,V)>0$ depends only on $q$ and $V$. Of course, as usual, the maximum could be replaced by the third largest number (as in \cite{BG06}) or even by the minimum as in \cite{BG21}. Nevertheless, it seems that the "losses of derivatives" associated with such estimates are too big to hope to put \eqref{NLS} in Birkhoff normal form in low regularity. That is why, the previous almost global well-posedness results only deal with very smooth solutions to \eqref{NLS}.

In this paper, we take advantage of the fact that it is enough to have small divisor estimates when $(\boldsymbol{k},\boldsymbol{\ell})$ satisfies \eqref{eq:monomials_to_remove} to draw potentials $V$ which are much less generic but which enjoy much better small divisors estimates. For simplicity\footnote{this choice could be easily generalized.}, we consider potentials of the form
\begin{equation}
\label{eq:def_V_rand}
V(x) = \frac1{(2\pi)^{d/2}} \sum_{k\in \mathbb{Z}^d} V_k e^{ik \cdot x} \quad \mathrm{where} \quad V_k = X_n, \quad n\ \mathrm{being \ the \ index\ such\ that} \  k\in B_n
\end{equation}
and $X_n \sim \mathcal{U}(0,1)$ are independent random variables uniformly distributed in $[0,1]$. Note that, this choice makes \eqref{NLS} partially resonant : generically the frequencies inside a same block are not rationally independent and so, a priori there should be energy exchanges inside blocks. Thanks to these multiplicities (in the values of $(V_k)_k$) we have actually much less small divisors to estimate and so we have  much better lower bounds on them. Roughly speaking, in the probability estimates, we do not have to make converge sums with respect to $k$ but only sums with respect to $n \simeq \log_2 k$. Therefore using standard estimates, we prove in Lemma \ref{lemma:nonres} that, almost surely,
\begin{equation}
\label{eq:new_SD_est}
|\Omega(\boldsymbol{k},\boldsymbol{\ell})| \geq \gamma(q,V)\, \big(\log \max_{1\leq j\leq q} (\langle \boldsymbol{k}_j \rangle, \langle \boldsymbol{\ell}_j \rangle)\big)^{-(2q+1)} \quad \mathrm{whenever} \quad (\boldsymbol{k}, \boldsymbol{\ell}) \quad \mathrm{satisfies \ \eqref{eq:monomials_to_remove}.}
\end{equation}
We point out that the gain between the standard small divisor estimates \eqref{eq:standard_SD_est} and the new ones \eqref{eq:new_SD_est} is huge: we have replaced polynomial losses of derivatives by logarithmic ones. Moreover, using technics inspired by \cite{Del09,BFGI21},
these logarithms losses can almost be considered as constants. So it means that we have no losses in our small divisor estimates. It is the main novelty of this paper and the reason why we can prove the almost global well-posedness of \eqref{NLS} in low regularity (i.e. Theorem \ref{th-main}).

The rest of the proof is quite classical but contains some technicalities mainly due to the three following facts : 
\begin{itemize}
\item since we work in low regularity, it is harder to justify some standard formal computations. In particular, we approximate the non-smooth solutions by smooth solutions in order to prove that if $\tau^{(0)}$ is the change of variable associated with the Birkhoff normal form procedure then $v(t) := \tau^{(0)}(u(t))$ is time derivable (see \eqref{eq:newynamic} for details).
\item since we prove a Nekhoroshev result (i.e. stability for exponentially long times), we have to optimize the order of the normal form with respect to the size of the solution and so to track all the constants carefully.
\item the logarithmic losses associated with the small divisors involve the largest index (see \eqref{eq:new_SD_est}). They can be seen as logarithmic losses of derivatives at each step of the Birkhoff normal form procedure. Therefore, a priori, the Hamiltonian flows we have to introduce cannot be simply defined by a fix point argument. To overcome this technicality we introduce a truncation in the spirit of \cite{BFGI21} (see subsection \ref{sec:loglosses} for details).
\end{itemize}

\noindent
{\bf Notation.} 
We shall use the notation $A\lesssim B$ to denote $A\le C B$ where $C$ is a positive constant
depending on  parameters fixed once for all, 
for instance $d$ and $p$. 
We will emphasize by writing $\lesssim_{s}$ 
when the constant $C$ 
depends on some other parameter $s$.

\subsection{Acknowledgments} During the preparation of this work the authors benefited from the support of the Centre Henri Lebesgue ANR-11-LABX-0020-0 and J.B. was also supported by the region "Pays de la Loire" through the project "MasCan".

\section{Hamiltonian formalism and Birkhoff normal form}
\label{sec:Hamform}
The results and formalisms of this section are standard and quite similar to the ones of \cite{Bam03,BG06,BG21}. Nevertheless, since we aim at proving a Nekhoroshev result, we have to track carefully the constants and, since we work in low regularity, we have to pay attention to justify the formal computations.

\subsection{Functional setting} We use the standard functional setting to deal with Hamiltonian systems. Nevertheless to avoid any possible confusion we recall it precisely (and we refer to section 3.1 of \cite{BG21} for further details and comments).

We always identity any function $u\in L^1(\mathbb{T}^d;\mathbb{C})$ with the sequence of its Fourier coefficients
$$
u_k = (2\pi)^{-d/2} \int_{\mathbb{T}^d} u(x) e^{-i k \cdot x} \mathrm{d}x, \quad k\in \mathbb{Z}^d
$$
We also naturally extend this definition to any distribution $u \in \mathcal{D}'(\mathbb{T}^d;\mathbb{C})$ by continuity. With such a convention, for all $s \in \mathbb{R}$ and all $u\in H^s(\mathbb{T}^d;\mathbb{C})$, we have
$$
\|u\|_{H^s}^2 = \sum_{k\in \mathbb{Z}^d} \langle k \rangle^{2s} |u_k|^2
$$
and for $u\in L^2(\mathbb{T}^d;\mathbb{C})$, the Fourier inversion formula reads
$$
u(x) = (2\pi)^{-d/2} \sum_{k\in \mathbb{Z}^d} u_k e^{ik\cdot x}.
$$
We always consider $L^2(\mathbb{T}^d;\mathbb{C}) = H^0(\mathbb{T}^d;\mathbb{C})$ as a real vector space. So it is naturally equipped with the following scalar product
$$
\forall u,v\in L^2, \quad (u,v)_{L^2} := \Re \int_{\mathbb{T}^d} u(x) \overline{v(x)} \, \mathrm{d}x =  \Re \sum_{k\in \mathbb{Z}^d}  u_k \overline{v_k}.
$$
Identifying distributions with their Fourier coefficients, we also equip $\mathcal{D}'(\mathbb{T}^d;\mathbb{C})$ with the discrete $\ell^p$  norms $p\geq 1$,
$$
\| u \|_{\ell^p}^p = \sum_{k\in \mathbb{Z}^d} |u_k|^p \quad \mathrm{and} \quad \| u \|_{\ell^\infty} := \sup_{k\in \mathbb{Z}^d} |u_k|.
$$
Being given $s \in \mathbb{R}$, we define the $\ell^1_s$ norm by
$$
\| u\|_{\ell^1_s} :=  \sum_{k\in \mathbb{Z}^d} \langle k \rangle^s |u_k|. 
$$
As usual we extend this scalar product when $u\in H^s(\mathbb{T}^d;\mathbb{C})$ and $v\in H^{-s}(\mathbb{T}^d;\mathbb{C})$.
Being given a smooth function $P :\ell^1(\mathbb{Z}^d;\mathbb{C}) \to \mathbb{R}$ and $u\in \ell^1$, its gradient $\nabla P(u)$ is the unique element of $\ell^\infty(\mathbb{Z}^d;\mathbb{C})$ satisfying
$$
\forall v\in \ell^1(\mathbb{Z}^d;\mathbb{C}), \ (\nabla P(u),v)_{L^2} = \mathrm{d}P(u)(v).
$$
Note that it can be checked that
$$
\forall k \in \mathbb{Z}^d, \quad (\nabla P(u))_k = 2\partial_{\overline{u_k}} P(u).
$$
We equip $L^2(\mathbb{T}^d;\mathbb{C})$ of the usual symplectic form $(i\cdot,\cdot)_{L^2}$.  Therefore a smooth map $\tau : \Omega \to \ell^1(\mathbb{Z}^d;\mathbb{C})$, where $\Omega$ is an open set of $ \ell^1(\mathbb{Z}^d;\mathbb{C})$, is \emph{symplectic} if
$$
\forall u \in \Omega, \forall v,w \in  \ell^1(\mathbb{Z}^d;\mathbb{C}), \quad (iv,w)_{L^2} = (i\mathrm{d}\tau(u)(v),\mathrm{d}\tau(u)(w))_{L^2}.
$$
Moreover, if $P,Q :  \ell^1(\mathbb{Z}^d;\mathbb{C}) \to \mathbb{R}$ are two functions such that $\nabla P$ is  $ \ell^1(\mathbb{Z}^d;\mathbb{C})$ valued then the \emph{Poisson bracket} of $P$ and $Q$ is defined by
$$
\{  P,Q\}(u):= (i \nabla P(u),\nabla Q(u))_{L^2}.
$$
Note that, as usual, we have
\begin{equation}
\label{eq:Poisson_formula}
\{  P,Q\}  
  =2i \sum_{k\in \mathbb{Z}^d} \partial_{\overline{u_k}}P(u) \partial_{u_k} Q(u) - \partial_{u_k}P(u) \partial_{\overline{u_k}} Q(u).
\end{equation}

\subsection{A class of homogeneous polynomials}
In this section, we aim at establishing the main properties of the class of Hamiltonians defined just below. The two main results are Proposition \ref{prop:Poisson} in which we prove its stability by Poisson bracket and Proposition \ref{prop:Lie_transform} in which we study their Hamiltonian flows.
\begin{definition}[homogeneous polynomials]
\label{def:class_ham} For $q\geq 2$, let $\mathscr{H}_{2q}$ be the space of the homogeneous formal polynomials of degree $2q$ of the form
$$
P(u) =  \sum_{\boldsymbol{k},\boldsymbol{\ell}\in (\mathbb{Z}^d)^q } P_{\boldsymbol{k},\boldsymbol{\ell} } \, u_{\boldsymbol{k}_1} \dots u_{\boldsymbol{k}_q} \overline{u_{\boldsymbol{\ell}_1}} \dots \overline{u_{\boldsymbol{\ell}_q}}
$$
with $P_{\boldsymbol{k},\boldsymbol{\ell} }  \in \mathbb{C}$, satisfying the reality condition
\begin{equation}
\label{eq:def_real}
P_{\boldsymbol{\ell} , \boldsymbol{k}}  = \overline{P_{\boldsymbol{k},\boldsymbol{\ell} } }
\end{equation}
the symmetry condition 
\begin{equation}
\label{def:sym_cond}
\forall \phi,\sigma \in \mathfrak{S}_q, \quad P_{\phi \boldsymbol{k},\sigma\boldsymbol{\ell} } =  P_{\boldsymbol{k},\boldsymbol{\ell} }
\end{equation}
the zero momentum condition
\begin{equation}
\label{def:zeromomcond}
 P_{\boldsymbol{k},\boldsymbol{\ell} } \neq 0 \quad \Longrightarrow \quad \boldsymbol{k}_1 + \cdots + \boldsymbol{k}_q = \boldsymbol{\ell}_1 + \cdots + \boldsymbol{\ell}_q
\end{equation}
and the bound
\begin{equation}
\label{eq_norm_ham}
\| P \|_{\ell^\infty} = \! \! \! \sup_{ \boldsymbol{k},\boldsymbol{\ell}\in (\mathbb{Z}^d)^q } | P_{ \boldsymbol{k},\boldsymbol{\ell} } | < \infty. 
\end{equation}
\end{definition}

\begin{lemma}
\label{lem:val_ham} The formal homogeneous polynomials define naturally smooth real valued functions on $\ell^1(\mathbb{Z}^d)$. More quantitatively, if $q\geq 2$, $P\in \mathscr{H}_{2q}$ and $u^{(1)},\dots,u^{(2q)} \in \ell^1(\mathbb{Z}^d)$ we have 
 \begin{equation}
 \label{eq:bebete}
 \sum_{\boldsymbol{h}\in (\mathbb{Z}^d)^{2q} } | P_{\boldsymbol{h} }  u_{\boldsymbol{h}_1}^{(1)} \dots u_{\boldsymbol{h}_{2q}}^{(2q)}  | \leq \| P \|_{\ell^\infty}  \prod_{j=1}^{2q} \|u^{(j)} \|_{\ell^1}
 \end{equation}
 In other words, the multi-linear map naturally associated with $P$ is well defined and continuous on $\ell^1$.
\end{lemma}
\begin{proof} 
 The estimate \eqref{eq:bebete} is a direct consequence of the zero momentum condition \eqref{def:zeromomcond}.
 Using the reality condition \eqref{eq:def_real}, it is straightforward to check that $P$ is real valued.
\end{proof}

\begin{corollary} 
\label{cor:permute_ok}
We can permute derivatives with the sum defining $P$.
\end{corollary}
\begin{proof} It is a classical corollary of the continuity of the multi-linear maps associated with $P$.
\end{proof}

\begin{corollary}
\label{cor_uniq}
If $P\in \mathscr{H}_{2q}$ vanishes everywhere on $\ell^1$ then $P=0$ (i.e. all its coefficients vanish).
\end{corollary}
\begin{proof}
It follows from the symmetry condition \eqref{def:sym_cond} and Corollary \ref{cor:permute_ok} that we have
$$
 P_{\boldsymbol{k},\boldsymbol{\ell} } \approx_{\boldsymbol{k},\boldsymbol{\ell} } \partial_{u_{\boldsymbol{k}_1}} \cdots \partial_{u_{\boldsymbol{k}_q}} \partial_{\overline{u_{\boldsymbol{\ell}_1}}} \cdots \partial_{\overline{u_{\boldsymbol{\ell}_q}}}    P(0) = 0.
$$
\end{proof}

Now, in the following proposition, we focus on the stability of the class by Poisson brackets.
\begin{proposition}\label{prop:Poisson} Let $q,q'\geq 2$, $P\in \mathscr{H}_{2q}$ and $Q\in \mathscr{H}_{2q'}$ be two homogeneous polynomials. Then, there exists a unique $S \in \mathscr{H}_{2(q+q'-1)}$ such that for all $u\in \ell^1(\mathbb{Z}^d)$, we have
$$
\{ P,Q\}(u) = S(u).
$$
Moreover, we have the estimate
$$
\| S \|_{\ell^\infty} \leq 4 \, q \, q' \| P \|_{\ell^\infty} \|Q\|_{\ell^\infty}.
$$
\end{proposition}
\begin{proof}
Thanks to the symmetry condition \eqref{def:sym_cond} and the zero momentum condition \eqref{def:zeromomcond}, we note that for all $u\in \ell^1$ and $m\in \mathbb{Z}^d$
$$
\partial_{\overline{u_m}}P(u) = q  \sum_{\boldsymbol{k}_1 + \cdots + \boldsymbol{k}_q = \boldsymbol{\ell}_1 + \cdots + \boldsymbol{\ell}_{q-1} +m    } P_{\boldsymbol{k},\boldsymbol{\ell},m } \, u_{\boldsymbol{k}_1} \dots u_{\boldsymbol{k}_q} \overline{u_{\boldsymbol{\ell}_1}} \dots \overline{u_{\boldsymbol{\ell}_{q-1}}}
$$
and that a similar formula holds for $\partial_{u_m}P(u)$.
Therefore, as usual, we deduce that
\begin{equation*}
\begin{split}
\{  P,Q\}(u)
  &=2i q q'\sum_{m\in \mathbb{Z}^d} \partial_{\overline{u_m}}P(u) \partial_{u_m} Q(u) - \partial_{u_m}P(u) \partial_{\overline{u_m}} Q(u)\\
  &= \sum_{\boldsymbol{k}_1 + \cdots + \boldsymbol{k}_{q''} = \boldsymbol{\ell}_1 + \cdots + \boldsymbol{\ell}_{q''} } R_{\boldsymbol{k},\boldsymbol{\ell}}  u_{\boldsymbol{k}_1} \dots u_{\boldsymbol{k}_{q''}}  \overline{u_{\boldsymbol{\ell}_1}} \dots \overline{u_{\boldsymbol{\ell}_{q''}}}
\end{split}  
  \end{equation*}
  where we have set $q''=q+q'-1$, $R_{\boldsymbol{k},\boldsymbol{\ell}} = 2i q q'T_{\boldsymbol{k},\boldsymbol{\ell}}$,
  \begin{equation*}
  \begin{split}
  T_{\boldsymbol{k},\boldsymbol{\ell}}  =& P_{\boldsymbol{k}_1,\cdots,\boldsymbol{k}_q ,\boldsymbol{\ell}_1, \cdots,\boldsymbol{\ell}_{q-1} ,\boldsymbol{k}_1 + \cdots + \boldsymbol{k}_q -   \boldsymbol{\ell}_1 -\cdots - \boldsymbol{\ell}_{q-1}} Q_{\boldsymbol{k}_{q+1},\cdots,\boldsymbol{k}_{q''} ,  \boldsymbol{\ell}_q +\cdots + \boldsymbol{\ell}_{q''} -   \boldsymbol{k}_{q+1} -\cdots - \boldsymbol{k}_{q''} ,\boldsymbol{\ell}_q, \cdots,\boldsymbol{\ell}_{q''}} \\
  &- P_{\boldsymbol{k}_1,\cdots,\boldsymbol{k}_{q-1} ,  \boldsymbol{\ell}_1 +\cdots + \boldsymbol{\ell}_{q} -   \boldsymbol{k}_1 -\cdots - \boldsymbol{k}_{q-1} ,\boldsymbol{\ell}_1, \cdots,\boldsymbol{\ell}_{q}} Q_{\boldsymbol{k}_q,\cdots,\boldsymbol{k}_{q''} ,\boldsymbol{\ell}_{q+1}, \cdots,\boldsymbol{\ell}_{q''} ,\boldsymbol{k}_q + \cdots + \boldsymbol{k}_{q''} -   \boldsymbol{\ell}_{q+1} -\cdots - \boldsymbol{\ell}_{q''}}
  \end{split}  
  \end{equation*}
  and the converge of the series is ensured by Lemma \ref{lem:val_ham}. The coefficients $R_{\boldsymbol{k},\boldsymbol{\ell}}$ satisfy clearly the reality condition and enjoy the bound 
  $$
  |R_{\boldsymbol{k},\boldsymbol{\ell}}| \leq 4 \, q \, q' \| P \|_{\ell^\infty} \|Q\|_{\ell^\infty}.
  $$
   Moreover, they can be extended by zero in such a way that they enjoy the zero momentum condition \eqref{def:zeromomcond}. However, a priori, they do not satisfy the symmetry condition \eqref{def:sym_cond}, so we just have to set
  $$
  S_{\boldsymbol{k},\boldsymbol{\ell}} := ((q'')!)^{-2} \sum_{\phi,\sigma \in \mathfrak{S}_{q''}} R_{\phi \boldsymbol{k},\sigma \boldsymbol{\ell}} .
  $$
\end{proof}

Now, we are going to estimate the vector fields these Hamiltonians generate.
\begin{lemma}
\label{lem:tech_est_ell1} Let $q\geq 2$ and $P\in \mathscr{H}_{2q}$ be a homogeneous formal polynomial of degree $2q$. Then, if $u^{(1)},\dots,u^{(2q-1)} \in \ell^1(\mathbb{Z}^d)$ and $w \in \ell^\infty(\mathbb{Z}^d)$, we have 
 $$
 \sum_{\boldsymbol{h} \in (\mathbb{Z}^d)^{2q} } | P_{\boldsymbol{h} }  u_{\boldsymbol{h}_1}^{(1)} \dots u_{\boldsymbol{h}_{2q-1}}^{(2q-1)} w_{\boldsymbol{h}_{2q}} | \leq  \| P \|_{\ell^\infty} \|w \|_{\ell^\infty}    \prod_{1\leq j \leq 2q-1 }  \|u^{(k)} \|_{\ell^1}.
 $$
\end{lemma}
\begin{proof}
 This estimate is still a direct consequence of the zero momentum condition \eqref{def:zeromomcond}.
\end{proof}
As a consequence, we get the following corollary  directly by duality.
\begin{corollary}  \label{cor:est_vfl1} Let $q\geq 2$, $P\in \mathscr{H}_{2q}$ be a homogeneous formal polynomial of degree $2q$. Then for all $u\in \ell^1$, $\nabla P(u) \in \ell^1$ and we have the estimates
$$
\| \nabla P(u) \|_{\ell^1} \leq 2q \| P \|_{\ell^\infty} \| u\|_{\ell^1}^{2q-1}.
$$
 Moreover, the map $\nabla P : \ell^1 \to  \ell^1$ is smooth and locally Lipschitz :
\begin{equation*}
\forall v\in \ell^1, \quad \| \mathrm{d}\nabla P(u)(v) \|_{ \ell^1 } \leq (2q)^2 \| P \|_{\ell^\infty} \| u\|_{\ell^1}^{2q-2}\| v\|_{\ell^1}.
\end{equation*}
\end{corollary}

As a consequence, we are in position to study the existence of Hamiltonian flows in $\ell^1(\mathbb{Z}^d)$.
\begin{proposition}[Lie transform] \label{prop:Lie_transform}Let $q\geq 2$, $\chi \in \mathscr{H}_{2q}$ and 
$$
\varepsilon_\chi  := \frac14 (2q \| \chi \|_{\ell^\infty})^{-\frac1{2q-2}}.
$$
Then there exists $C^{\infty}$ map $\Phi_{\chi} : [-1,1] \times B_{\ell^1}(0, \varepsilon_\chi ) \to \ell^1(\mathbb{Z}^d)$ such that if $|t| \leq 1$ and $\| u \|_{ \ell^1 } < \varepsilon_\chi$, we have
\begin{equation}
\label{eq:cestleflot}
-i\partial_t \Phi_{\chi}^t(u) = \nabla \chi( \Phi_{\chi}^t(u)) \quad \mathrm{and} \quad  \Phi_{\chi}^0(u) = u.
\end{equation}
Moreover, being given $\| u \|_{ \ell^1 } < \varepsilon_\chi$ and $|t|\leq 1$, it enjoys the following properties :
\begin{enumerate}[i)]
\item $\Phi_{\chi}^t$ is symplectic :
$$
\forall v,w \in \ell^1, \quad (iv,w)_{L^2} = (i\mathrm{d}\Phi_{\chi}^t(u)(v),\mathrm{d}\Phi_{\chi}^t(u)(w))_{L^2}.
$$
\item $\Phi_{\chi}^t$ is invertible :
\begin{equation}
\label{eq:invertPhichi}
\|  \Phi_{\chi}^t(u)  \|_{ \ell^1 } < \varepsilon_\chi  \quad \Rightarrow \quad \Phi_{\chi}^{-t} ( \Phi_{\chi}^{t}(u)) =u.
\end{equation}
\item $\Phi_{\chi}^t$ is close to the identity : 
\begin{equation}
\label{eq:closeidl1}
\| \Phi_{\chi}^t(u) -u \|_{\ell^1} \leq  \left( \frac{\| u\|_{\ell^1}}{\varepsilon_\chi} \right)^{2q-2} \| u\|_{\ell^1}.
\end{equation}
\item $\Phi_{\chi}^t$ is locally Lipschitz :
\begin{equation}
\label{eq:closeiddiffl1}
\forall v\in \ell^1, \quad \| \mathrm{d} \Phi_{\chi}^t(u)(v) \|_{\ell^1} \leq 2 \|v\|_{\ell^1}.
\end{equation}
\end{enumerate}
\end{proposition}
\begin{proof} Since, by Corollary \ref{cor:est_vfl1}, the vector field $i\nabla \chi$ is locally-Lipschitz, the local existence and the smoothness of the flow $\Phi_{\chi}^t$ is ensured by the Cauchy-Lipchitz Theorem. The only thing we have to check is that the solutions exist for $|t|\leq 1$. Without loss of generality we only consider positive times. More precisely, let $T>0$ and $v\in C^1([0,T); \ell^1)$ be a solution of the Cauchy problem
$$
-i\partial_t v(t) = \nabla \chi(v(t)) \quad \mathrm{and} \quad v(0) = u \in B_{\ell^1}(0, \varepsilon_\chi ).
$$
It is enough to prove that if $u\neq 0$, $0\leq t< T$ and $t \leq 1$ then $\| v(t) \|_{\ell^1} \leq  2 \| u \|_{\ell^1} < 2 \varepsilon_\chi$. We set $I = [0,T) \cap [0,1]$ and we aim at proving that $S=I$ where
$$
S = \{ t \in I \ | \ \forall \tau \in [0,t], \ \|v(\tau)\| \leq  2 \| u \|_{\ell^1} \}.
$$
Since $v$ is continuous, $S$ is clearly non-empty and closed in $I$. Moreover, if $t\in I$ then
\begin{multline}
\label{eq:tresrelou}
\| v(t) - u \|_{\ell^1} \leq \int_0^t \| \nabla \chi(v(\tau)) \|_{\ell^1} \mathrm{d}\tau \leq  2q t \| \chi \|_{\ell^\infty}  (2\| u\|_{\ell^1})^{2q-1} \leq 2^{2q-1}  \left( \frac{\| u\|_{\ell^1}}{4\varepsilon_\chi} \right)^{2q-2} \| u\|_{\ell^1} \\
\leq 2^{-2q + 3}  \left( \frac{\| u\|_{\ell^1}}{\varepsilon_\chi} \right)^{2q-2} \| u\|_{\ell^1} 
\end{multline}
and so $\| v(t) \|_{\ell^1} \leq (1+2^{-2q + 3} ) \| u \|_{\ell^1} < 2 \| u \|_{\ell^1} $. Therefore, since $v$ is continuous, $S$ is open and so, since $I$ is connected, we have $S = I$.

Now, that we have checked the existence of $\Phi_\chi$, we focus on properties $i),ii)$ and $iii)$. First, the property $ii)$ is ensured by the fact that $\Phi_{\chi}$ is a flow.  Moreover the property $iii)$ has been proven in \eqref{eq:tresrelou}. Finally, since $\Phi_\chi$ is a Hamiltonian flow, it is standard to check $i)$ (i.e. that $\Phi_{\chi}^t$ is symplectic).

Finally, we focus on $iv)$. If $w\in \ell^1$, we have
\begin{equation}
\label{eq:ladernierejespere}
-i\partial_t \mathrm{d} \Phi_\chi^t(u)(w) = \mathrm{d} \nabla \chi (v(t))( \mathrm{d} \Phi_\chi^t(u)(w))
\end{equation}
and thus
\begin{equation*}
\begin{split}
\|  \mathrm{d} \Phi_\chi^t(u)(w) -w \|_{\ell^1} &\leq \int_{0}^t \|  \mathrm{d} \nabla \chi (v(\tau))( \mathrm{d} \Phi_\chi^{\tau}(u)(w)) \|_{\ell^1}\mathrm{d} \tau \\
&\leq  (2q)^2  \| \chi \|_{\ell^\infty}  \int_{0}^t  \| v(\tau)\|_{\ell^1}^{2q-2} \|  \mathrm{d} \Phi_\chi^{\tau}(u)(w) \|_{\ell^1}\mathrm{d} \tau \\
&\leq \left( \frac{\| u\|_{\ell^1}}{\varepsilon_\chi} \right)^{2q-2}  \int_{0}^t  \|  \mathrm{d} \Phi_\chi^{\tau}(u)(w) \|_{\ell^1}\mathrm{d} \tau.
\end{split}
\end{equation*}
Therefore, as a consequence of Gr\"onwall's inequality, we get  \eqref{eq:closeiddiffl1}.
\end{proof}

Now, we aim at establishing $H^s$ tame estimates.
\begin{lemma}[$H^s$ tame estimates]
\label{lem:tech_est} Let $s\geq 0$, $q\geq 2$ and $P\in \mathscr{H}_{2q}$ be a homogeneous formal polynomial of degree $2q$. Then, if $u^{(1)},\dots,u^{(2q-1)} \in \ell^1 \cap H^s$ and $w \in H^{-s}$, we have 
 $$
 \sum_{\boldsymbol{h} \in (\mathbb{Z}^d)^{2q} } | P_{\boldsymbol{h} }  u_{\boldsymbol{h}_1}^{(1)} \dots u_{\boldsymbol{h}_{2q-1}}^{(2q-1)} w_{\boldsymbol{h}_{2q}} | \leq (2q-1)^{(s-1)_+}  \| P \|_{\ell^\infty} \|w \|_{H^{-s}} \sum_{j=1}^{2q-1}  \|u^{(j)} \|_{H^s}   \prod_{k\neq j }  \|u^{(k)} \|_{\ell^1}
 $$
 where $(s-1)_+ = \max(s-1,0)$.
\end{lemma}
\begin{proof} We define $u^{(2q)}_k = \langle k \rangle^{-2s} w_{-k}$ in order to have $\|u^{(2q)} \|_{H^s} = \|w\|_{H^{-s}}$. Using the zero momentum condition, we have
$$
\mathcal{J}:=\sum_{\boldsymbol{h} \in (\mathbb{Z}^d)^{2q} } | P_{\boldsymbol{h} }  u_{\boldsymbol{h}_1}^{(1)} \dots u_{\boldsymbol{h}_{2q-1}}^{(2q-1)} w_{\boldsymbol{h}_{2q}}| \leq \| P \|_{\ell^\infty} \sum_{  \sigma_1  \boldsymbol{h}_1 + \cdots + \sigma_{2q}\boldsymbol{h}_{2q} =0 } | u_{\boldsymbol{h}_1}^{(1)} | \dots |u_{\boldsymbol{h}_{2q}}^{(2q)} |  \, \langle \boldsymbol{h}_{2q} \rangle^{2s}.
$$
where $\sigma_j = 1$ if $j \leq q$ and $\sigma_j = -1$ else.
Since, by Jensen (and the triangle inequality), we have
$$
\langle  \sigma_1  \boldsymbol{h}_1 + \cdots + \sigma_{2q-1}\boldsymbol{h}_{2q-1}   \rangle^{s} \leq (2q-1)^{(s-1)_+}(\langle \boldsymbol{h}_{1}\rangle^{s} + \cdots + \langle \boldsymbol{h}_{2q-1} \rangle^{s})
$$
where $(s-1)_+ = \max(s-1,0)$, applying the Young's inequality for convolutions, we deduce that
\begin{equation*}
\mathcal{J} \leq  \| P \|_{\ell^\infty} (2q-1)^{(s-1)_+}  \| u_{\boldsymbol{h}_{2q}}^{(2q)} \|_{H^s}  \sum_{j=1}^{2q-1} \| u_{\boldsymbol{h}_{j}}^{(j)} \|_{H^s} \prod_{k\neq j} \| u_{\boldsymbol{h}_{k}}^{(k)} \|_{\ell^1}.
\end{equation*}
\end{proof}
As a consequence, the following corollary follows directly by duality (see \cite{BG21} for more details).
\begin{corollary}  \label{cor:est_vf} Let $q\geq 2$, $s\geq 0$ and $P\in \mathscr{H}_{2q}$ be a homogeneous formal polynomial of degree $2q$. Then for all $u\in \ell^1 \cap H^s$, $\nabla P(u) \in H^s$ and we have the estimates
\begin{equation}
\label{eq:mangeralafourmis}
\| \nabla P(u) \|_{H^s} \leq 2q (2q-1)^{1+(s-1)_+}  \| P \|_{\ell^\infty} \| u\|_{\ell^1}^{2q-2} \| u\|_{H^{s}}
\end{equation}
 Moreover, the map $\nabla P : \ell^1 \cap H^s \to H^s$ is smooth and locally Lipschitz : for all $v\in  \ell^1 \cap H^s$
\begin{equation}
\label{eq:mangeralacigalle}
\| \mathrm{d}\nabla P(u)(v) \|_{H^s} \leq 2q  (2q-1)^{2+(s-1)_+} \| P \|_{\ell^\infty} \| u\|_{\ell^1}^{2q-3} (\|u\|_{H^s}\| v\|_{\ell^1}+\| u\|_{\ell^1}\|v\|_{H^s}).
\end{equation}
\end{corollary}

Thus, we also get tame estimates for the Lie transforms.
\begin{proposition}[Tame estimates for Lie transforms] \label{prop:Lie_transformHs}Let $q\geq 2$, $\chi \in \mathscr{H}_{2q}$ and $s\geq 0$. Then, being given $t\in[-1,1]$ and $u\in \ell^1 \cap H^s$ such that $\| u \|_{ \ell^1 } < \varepsilon_\chi$, the map $\Phi_\chi  : [-1,1] \times B_{\ell^1}(0, \varepsilon_\chi ) \to \ell^1(\mathbb{Z}^d)$ given by Proposition \ref{prop:Lie_transform} enjoys the following properties :
\begin{enumerate}[i)]
\item $\Phi_\chi^t$ preserves the $H^s$ regularity : $\Phi_\chi^t(u) \in H^s$.
\item $\Phi_\chi^t$ is close to the identity in $H^s$
\begin{equation}
\label{eq:est_tamevf}
\| \Phi_{\chi}^t(u) -u \|_{H^s} \lesssim_s  \left( \frac{\| u\|_{\ell^1}}{\varepsilon_\chi} \right)^{2q-2} \| u\|_{H^s}.
\end{equation}
\item $\Phi_\chi^t$ is locally Lipschitz on $H^{s} \cap \ell^1$.
\begin{equation}
\label{eq:est_tamedvf}
\forall w\in  \ell^1 \cap H^s, \quad \| \mathrm{d} \Phi_\chi^t (u)(w) \|_{H^{s}} \lesssim_s \| w\|_{H^{s}} +\varepsilon_\chi^{-1} \|w\|_{\ell^1}\| u\|_{H^s}.
\end{equation}
\item $\Phi_\chi^t:  B_{\ell^1}(0, \varepsilon_\chi ) \cap H^s \to \ell^1 \cap H^s$ is smooth.
\end{enumerate}
\end{proposition}
\begin{proof}
Since we proven in Corollary \ref{cor:est_vf} and Corollary \ref{cor:est_vfl1} that $i\nabla \chi$ is smooth and locally Lipschitz on  $\ell^1 \cap H^s$, the local existence of the flow of the equation $-i\partial_t v(t) = \nabla \chi(v(t))$ in $\ell^1 \cap H^s$  is ensured by the Cauchy Lipschitz Theorem. Therefore to prove $i)$ and $iv)$ we just have to prove if $u\in \ell^1 \cap H^s$ satisfies $\| u \|_{ \ell^1 } < \varepsilon_\chi$ then $\|\Phi_\chi^t(u)\|_{H^s}$ remains bounded while $|t|\leq 1$. We recall that the existence of $\Phi_\chi^t(u)$ in $\ell^1$ for $|t|\leq 1$ is ensured by Proposition \ref{prop:Lie_transformHs}. Without loss of generality, we only consider positive times. By definition of $v(t) := \Phi_\chi^t(u)$, noticing that $ (2q-1)^{1+(s-1)_+} \lesssim_s 2^q$, we have
\begin{equation}
\label{eq:patience}
\begin{split}
\| v(t) - u \|_{H^s} \leq  \int_0^t \| \nabla \chi (v(\tau)) \|_{H^s} \mathrm{d}\tau  &\mathop{\lesssim_s}^{\eqref{eq:mangeralafourmis}} 2q2^q  \| \chi \|_{\ell^\infty}  \int_0^t   \| v(\tau)\|_{\ell^1}^{2q-2}  \| v(\tau) \|_{H^s} \mathrm{d}\tau  \\
&\mathop{\lesssim_s}^{\eqref{eq:closeidl1}}  2^q  (4\varepsilon_\chi)^{-(2q-2)}  \int_0^t  ( 2 \|u \|_{\ell^1})^{2q-2}  \| v(\tau) \|_{H^s} \mathrm{d}\tau \\
&\lesssim_s \left( \frac{\| u\|_{\ell^1}}{\varepsilon_\chi} \right)^{2q-2}  \int_0^t  \| v(\tau) \|_{H^s} \mathrm{d}\tau.
\end{split}
\end{equation}
Therefore, since $\| u\|_{\ell^1} < \varepsilon_\chi$ and  $v(0) = u$, it follows by Gr\"onwall that 
\begin{equation}
\label{eq:grown}
\| v(t) \|_{H^s} \leq e^{t C_s} \| u\|_{H^s} 
\end{equation}
where $C_s$ is a constant depending only on $s$. Therefore, we deduce that $ v(t) \in H^s$ for $t \in [-1,1]$ (i.e. the assertion $i)$), and plugging \eqref{eq:grown} into \eqref{eq:patience} that $\Phi_\chi^t$ is close to the identity (i.e. that $ii)$ holds).

If $w\in \ell^1 \cap H^s$, $\partial_t \mathrm{d} \Phi_\chi^t(u)(w)$ is solution to \eqref{eq:ladernierejespere},
thus we have
\begin{equation*}
\begin{split}
&\|  \mathrm{d} \Phi_\chi^t(u)(w) -w \|_{H^{s}} \leq \int_{0}^t \|  \mathrm{d} \nabla \chi (v(\tau))( \mathrm{d} \Phi_\chi^{\tau}(u)(w)) \|_{H^{s}}\mathrm{d} \tau \\
\mathop{\lesssim_s}^{\eqref{eq:mangeralacigalle}} & 2q 2^q  \| \chi \|_{\ell^\infty}  \int_{0}^t  \| v(\tau)\|_{\ell^1}^{2q-3} ( \| v(\tau)\|_{\ell^1} \|  \mathrm{d} \Phi_\chi^{\tau}(u)(w) \|_{H^{s}} +\| v(\tau)\|_{H^s} \|  \mathrm{d} \Phi_\chi^{\tau}(u)(w) \|_{\ell^1}  ) \mathrm{d} \tau.
\end{split}
\end{equation*}
Then we use that $\| v(\tau)\|_{\ell^1} \leq 2 \| u\|_{\ell^1}$ (see \eqref{eq:closeidl1}), $\| v(\tau)\|_{H^s} \lesssim_s  \| u\|_{H^s}$ (just proved in \eqref{eq:grown}) and $ \|  \mathrm{d} \Phi_\chi^{\tau}(u)(w) \|_{\ell^1}  \leq 2 \|w\|_{\ell^1}$ (see \eqref{eq:closeiddiffl1}) to get
$$
\|  \mathrm{d} \Phi_\chi^t(u)(w) -w \|_{H^{s}}  \lesssim_s    \varepsilon_\chi^{-1} \| u\|_{H^s} \|w\|_{\ell^1} + \int_{0}^t  \|  \mathrm{d} \Phi_\chi^{\tau}(u)(w) \|_{H^{s}}  \mathrm{d} \tau.
$$
Therefore, as a consequence of Gr\"onwall's inequality, we get \eqref{eq:est_tamedvf}.
\end{proof}

Finally, in the following proposition, we prove that if $Z$ is a quadratic integrable polynomial then $\mathrm{ad}_Z$ is diagonal (and so easy to invert).
\begin{proposition}\label{prop:Poisson_vs_Z} Let $q\geq 2$, $P\in \mathscr{H}_{2q}$ and $Z$ be a polynomial of the form
$
Z(u) = \sum_{k\in \mathbb{Z}^d} g_k |u_k|^2
$ 
where $g_k \in \mathbb{R}$ satisfies $| g_k | \lesssim \langle k \rangle^{2s}$ for some $s \geq 0$. Then, that for all  $u\in H^s \cap \ell^1$, we have\footnote{note that the convergence of this series is ensured by Proposition \ref{lem:tech_est}.}
$$
\{ P,Z\}(u) = -2i \sum_{\boldsymbol{k},\boldsymbol{\ell} \in (\mathbb{Z}^d)^q} ( g_{\boldsymbol{k}_1}+ \cdots+g_{\boldsymbol{k}_q} - g_{\boldsymbol{\ell}_1} - \cdots - g_{\boldsymbol{\ell}_q}   ) P_{\boldsymbol{k},\boldsymbol{\ell} } u_{\boldsymbol{k}_1} \dots u_{\boldsymbol{k}_n} \overline{u_{\boldsymbol{\ell}_1}} \dots \overline{u_{\boldsymbol{\ell}_q}}.
$$
\end{proposition}

\begin{proof} This proposition is standard, we refer for example to \cite[Lemma 3.14]{BG21} for a detailed proof in a similar setting.
\end{proof}

\subsection{Birkhoff normal form}
\label{subsec:Birk} 

In this section, we consider a Hamiltonian system of the form
\begin{equation}
\label{eq:defH}
H = Z_2 + P
\end{equation}
where $P\in \mathscr{H}_{2p+2}$ ($p\geq 1$ is a number given by \eqref{NLS}) stands for the nonlinear part of the system and $Z_2$ is a quadratic Hamiltonian of the form
$$
Z_2(u) = \sum_{k\in \mathbb{Z}^d} \omega_k |u_k|^2.
$$
The \emph{frequencies} $\omega_k \in \mathbb{R}$ are real numbers. We assume for convenience that $|\omega| \lesssim \langle k \rangle^2$. Therefore, by Lemma \ref{lem:val_ham}, if $s>\max(1,d/2)$, $H$ is a smooth function on $H^s(\mathbb{T}^d;\mathbb{C})$.

\begin{definition}[small divisors]\label{def:SD} The small divisors, $\Omega(\boldsymbol{k},\boldsymbol{\ell}) $, are defined for $\boldsymbol{k},\boldsymbol{\ell} \in (\mathbb{Z}^d)^q$, $q\geq 2$, by
$$
\Omega(\boldsymbol{k},\boldsymbol{\ell})  =2i( \omega_{\boldsymbol{k}_1} + \cdots +  \omega_{\boldsymbol{k}_q} -  \omega_{\boldsymbol{\ell}_1} - \cdots -  \omega_{\boldsymbol{\ell}_q}).
$$

\end{definition}

\begin{definition}[resonance] \label{def:resHam} Being given $\nu > 0$, a homogeneous polynomial $P \in \mathscr{H}_{2q}$, with $q\geq 2$, is $\nu$-resonant (resp. $\nu$-nonresonant) if 
$$
\forall \boldsymbol{k},\boldsymbol{\ell} \in (\mathbb{Z}^d)^q, \quad |\Omega(\boldsymbol{k},\boldsymbol{\ell})| \geq \nu \ \Rightarrow \ P_{\boldsymbol{k},\boldsymbol{\ell}} = 0
$$
(resp. $|\Omega(\boldsymbol{k},\boldsymbol{\ell})| < \nu \ \Rightarrow \ P_{\boldsymbol{k},\boldsymbol{\ell}} = 0$). We denote by $\mathscr{H}^{(\nu-\mathrm{res})}_{2q}$ (resp. $\mathscr{H}^{(\nu- \mathrm{nonres})}_{2q}$) the real vector space they generate. 

\end{definition}

In this section, we aim at proving the following theorem.
\begin{theorem}[Birkhoff normal form] \label{thm:Birk} Let $H$ be the Hamiltonian given by \eqref{eq:defH} and $\nu \in (0,1)$. There exists a constant $C>1$, depending only on $P$, such that for all $r\geq 2$, setting
$$
\rho = \frac{\sqrt{\nu}}{Cr},
$$
there exist two $C^\infty$ symplectic maps $\tau^{(0)}$ and $\tau^{(1)}$ making the following diagram to commute
\begin{equation}
\label{mon_beau_diagram_roi_des_forets}
\xymatrixcolsep{5pc} \xymatrix{  B_{\ell^1}(0, \rho )  \ar[r]^{ \tau^{(0)} }
 \ar@/_1pc/[rr]_{ \mathrm{id}_{\ell^1} } &  B_{\ell^1}(0,2\, \rho)   \ar[r]^{ \hspace{1cm} \tau^{(1)} }  & 
 \ell^1(\mathbb{Z}^d) }
\end{equation}
such that on $B_{\ell^1}(0,2\rho)$, $H \circ \tau^{(1)}$ admits the decomposition
\begin{equation}\label{eq:NF}
H \circ \tau^{(1)} = Z_2+ \sum_{q=2}^{r} L^{(2q)}+ R
\end{equation}
where $L^{(2q)} \in \mathscr{H}^{(\nu-\mathrm{res})}_{2q}$ is a $\nu$-resonant homogeneous polynomial of degree $2q$  satisfying
$$
\| L^{(2q)}\|_{\ell^\infty} \leq  C^{2q} \Big(\frac{q^2}{\nu} \Big)^{q-2}
$$
 and $R : B_{\ell^1}(0,2\rho) \to \mathbb{R}$ is a $C^\infty$ function which is a remainder term of order $2r+2$ : for all $s\geq 0$ and all $u\in B_{\ell^1}(0,2\rho) \cap H^s$, 
\begin{equation}
\label{eq:estRemfinalBirk}
\|\nabla R(u)\|_{H^s} \lesssim_s C^{2r} \Big(\frac{r^3}{\nu} \Big)^{r-1}  \|u\|_{\ell^1}  ^{2r}  \|u\|_{H^s}.
\end{equation}
Moreover, for all $s\geq 0$ and they do not make the $H^s$ norm increase too much: if $u\in \ell^1 \cap H^s$ satisfies $\|u\|_{\ell^1} < 2^\sigma \rho$ with $\sigma \in \{0,1\}$ then
\begin{equation}
\label{eq:inflationHs}
\|\tau^{(\sigma)}(u) \|_{H^s} \lesssim_s \| u\|_{H^s}.
\end{equation}
Furthermore, the map $\tau^{(\sigma)} : B_{\ell^1}(0, 2^{\sigma}\rho ) \cap H^s \to \ell^1 \cap H^s $ is smooth.
\end{theorem}
\begin{proof} \emph{\underline{$\triangleright$ Step 1 : Setting of the induction.}} Let $C_2>1$ be the constant depending only $P$ such that the estimates \eqref{eq:C21},\eqref{eq:C22} below hold. We are going to prove by induction that for all $\mathfrak{r} \in \llbracket 1,r\rrbracket$, setting 
$$
\rho :=\frac{ \sqrt{\nu} }{112 C_2 \mathfrak{r}},
$$ there exist two $C^\infty$ symplectic maps $\tau^{(0)}$ and $\tau^{(1)}$ making the diagram \eqref{mon_beau_diagram_roi_des_forets} to commute and such that on $B_{\ell^1}(0,2\rho)$, $H \circ \tau^{(1)}$ admits the decomposition \eqref{eq:NF} where \begin{center}
$L^{(2q)} \in \mathscr{H}_{2q}$ is $\nu$-resonant for $q\leq \mathfrak{r}$
\end{center}
 and satisfies the estimate
\begin{equation}
\| L^{(2q)}\|_{\ell^\infty} \leq C_2^{2q - 3} \nu^{-q+2} \min(q,\mathfrak{r})^{2(q-2)} \quad \mathrm{for} \quad 2\leq q\leq r.
\end{equation}
Moreover the remainder term $R : B_{\ell^1}(0,2\rho) \to \mathbb{R}$ is a $C^\infty$ map satisfying, for all $u\in B_{\ell^1}(0,2\rho)$,
\begin{equation}
\label{eq:estRemindl1}
\|\nabla R(u)\|_{\ell^1} \leq K_{\ell^1}^{\mathfrak{r}} (2^5 C_2^2 \nu^{-1} r^{2})^{r-1} \big( \prod_{j=1}^{\mathfrak{r}-1} (1+2^{-2j} ) \big)^{2r}  \| u \|_{ \ell^1}^{2r+1} 
\end{equation}
where $K_{\ell^1}>1$ is an universal constant given by \eqref{eq:defKell1}, and for all $s\geq 0$ and all $u\in B_{\ell^1}(0,2\rho) \cap H^s$, 
\begin{equation}
\label{eq:estRemind}
\|\nabla R(u)\|_{H^s} \leq K_s^{\mathfrak{r}} (2^5 C_2^2 \nu^{-1} r^{2})^{r-1} \big( \prod_{j=1}^{\mathfrak{r}-1} (1+2^{-2j} ) \big)^{2r}  \| u \|_{ \ell^1}^{2r} \|  u \|_{ H^s}
\end{equation}
where $K_s>1$ is the constant depending only on $s$ given by \eqref{eq:defK_s},
and if $u\in \ell^1 \cap H^s$ satisfies $\|u\|_{\ell^1} < 2^\sigma \rho$ with $\sigma \in \{0,1\}$ then we have
\begin{equation}
\label{eq:growthHs}
\|\tau^{(\sigma)}(u) \|_{H^s} \leq  \Big(\prod_{j=1}^{ \mathfrak{r}-1 } 1+M_s 2^{-j} \Big)\| u\|_{H^s} 
\end{equation}
where $M_s>0$ is the constant depending only on $s$ given by \eqref{eq:est_tamevf}, and
\begin{equation}
\label{eq:growthl1}
 \|\tau^{(\sigma)}(u) \|_{\ell^1} \leq \left( 1 +  \frac{\| u\|_{\ell^1}^2}{ 2^{2\sigma}\rho^2} \right) \| u\|_{\ell^1}.
\end{equation}

Note that, since the product in\eqref{eq:estRemind} and \eqref{eq:growthHs} are convergent and that $K_s^{\mathfrak{r}} \lesssim_s r^r$, this result  (when $\mathfrak{r}=r$)  is just refinement of Theorem \ref{thm:Birk}. The fact that $\tau^{(\sigma)} : B_{\ell^1}(0, 2^{\sigma}\rho ) \cap H^s \to \ell^1 \cap H^s$ is smooth is just a direct corollary of the construction and property $iv)$ of Proposition \ref{prop:Lie_transformHs}.

We are going to proceed by induction on $\mathfrak{r}$.  First, we note that the case $\mathfrak{r}=1$ is obvious, provided that $C_2$ is chosen large enough to ensure that
\begin{equation}
\label{eq:C21}
\| P \|_{\ell^\infty} \leq  C_2^{2 p - 1}
\end{equation}
and so\footnote{we recall that $p\geq 1$ and $\nu \in(0,1)$.} $\| P \|_{\ell^\infty} \leq C_2^{2 p - 1} \nu^{-p+1} \min(p+1,1)^{2(p-1)} $.

From now, we assume that the property holds at step $\mathfrak{r}<r$ and we are going to prove it at the step $\mathfrak{r}+1$. In order to get convenient notations, we denote with a subscript $\sharp$ the maps corresponding to the step $\mathfrak{r}+1$ (like for example $\tau^{(1),\sharp}$ or $R^\sharp$).

\medskip

\noindent \emph{\underline{$\triangleright$ Step 2 : Resolution of the cohomological equation.}} In order to remove the $\nu$-non resonant terms of $L^{(2\mathfrak{r}+2)}$, we define a polynomial Hamiltonian $\chi \in \mathscr{H}_{2\mathfrak{r}+2}$ by
$$
\chi_{\boldsymbol{k},\boldsymbol{\ell}} :=  \frac{L^{(2\mathfrak{r}+2)}_{\boldsymbol{k},\boldsymbol{\ell}}}{i \Omega(\boldsymbol{k},\boldsymbol{\ell})} \quad  \mathrm{if} \quad |\Omega(\boldsymbol{k},\boldsymbol{\ell})| \geq \nu \quad \mathrm{and} \quad \chi_{\boldsymbol{k},\boldsymbol{\ell}} = 0 \quad \mathrm{ else}.
$$
As a consequence of Proposition \ref{prop:Poisson_vs_Z} it can be easily checked that 
$$
 L^{(2\mathfrak{r}+2),\sharp}  = \{ \chi,Z_2\} + L^{(2\mathfrak{r}+2)} 
$$
where  $L^{(2\mathfrak{r}+2),\sharp} \in \mathscr{H}^{(\nu-\mathrm{res})}_{2\mathfrak{r}+2}$ is the $\nu$-resonant part of $L^{(2\mathfrak{r}+2)}$, i.e.
\begin{equation}
\label{eq:cestlapartienonresonnantedemachin}
 L^{(2\mathfrak{r}+2),\sharp}_{\boldsymbol{k},\boldsymbol{\ell}} = L^{(2\mathfrak{r}+2)}_{\boldsymbol{k},\boldsymbol{\ell}} \quad \mathrm{if} \quad |\Omega(\boldsymbol{k},\boldsymbol{\ell})| < \nu \quad \mathrm{and} \quad  L^{(2\mathfrak{r}+2),\sharp}_{\boldsymbol{k},\boldsymbol{\ell}} =0 \quad \mathrm{else}.
\end{equation}

\medskip

\noindent \emph{\underline{$\triangleright$ Step 3 : The new variables.}} The Hamiltonian $\chi$ clearly enjoys the bound 
$$
\|\chi \|_{\ell^\infty} \leq \nu^{-1 }\|  L^{(2\mathfrak{r}+2)} \|_{\ell^\infty} \leq C_2^{2\mathfrak{r} - 1} \nu^{-\mathfrak{r}} \mathfrak{r}^{2(\mathfrak{r} -1 )} .
$$
We recall that the Hamiltonian flow of $\chi$ is given by Proposition \eqref{prop:Lie_transform} and is well defined for $|t|\leq 1$ on $B_{\ell^1}(0, \varepsilon_\chi)$ where
\begin{equation}
\label{eq:valabon}
\varepsilon_\chi = \frac14 (2(\mathfrak{r}+1) \| \chi \|_{\ell^\infty})^{-\frac1{2 \mathfrak{r} }} \geq \frac14 (2(\mathfrak{r}+1) C_2^{2\mathfrak{r} - 1} \nu^{-\mathfrak{r}} \mathfrak{r}^{2(\mathfrak{r} -1 )} )^{-\frac1{2 \mathfrak{r} }} \geq \frac{\sqrt{\nu} }{8 \mathfrak{r} C_2} = 14 \rho.
\end{equation}
Now, we aim at defining
$$
\tau^{(1),\sharp} := \tau^{(1)} \circ \Phi_\chi^1 \quad \mathrm{and} \quad \tau^{(0),\sharp} := \Phi_\chi^{-1} \circ  \tau^{(0)}.
$$
So we have to check that these maps are well defined on $B_{\ell^1}(0,2 \rho^\sharp)$ and $B_{\ell^1}(0, \rho^\sharp)$ respectively. First, since $\rho^\sharp \leq \rho $, thanks to the estimate \eqref{eq:growthl1}, we know that  $\tau^{(0)}$ maps $B_{\ell^1}(0, \rho^\sharp)$ on $B_{\ell^1}(0, 2\rho^\sharp)$. Moreover since $2\rho^\sharp \leq 2 \rho \leq \varepsilon_\chi$ (see \eqref{eq:valabon}), it follows that $\tau^{(0),\sharp}$ is well defined on $B_{\ell^1}(0, \rho^\sharp)$. Then, in order to prove that $\tau^{(1),\sharp} $ is well defined, we just have to check that $\Phi_\chi^1$ maps $B_{\ell^1}(0, 2\rho^\sharp)$ on $B_{\ell^1}(0, 2\rho)$. Indeed, if $\|u\|_{\ell^1} < 2\rho^\sharp$ then (since $\Phi_\chi^1$ is close to the identity, see \eqref{eq:closeidl1})
\begin{equation}
\label{eq:cestlesvacances}
\begin{split}
\frac{\|\Phi_\chi^1 (u) \|_{\ell^1}}{2\rho} &\leq \left(1+ \left( \frac{ \| u \|_{\ell^1}  }{\varepsilon_\chi} \right)^{2 \mathfrak{r} } \right) \frac{\| u \|_{\ell^1} }{2\rho} \leq  \left(1+ \left( \frac{ \| u \|_{\ell^1}  }{14 \rho^\sharp } \right)^{2 \mathfrak{r} } \right) \frac{\| u \|_{\ell^1} }{2\rho^\sharp} \frac{\mathfrak{r}}{\mathfrak{r}+1} \\
&\leq (1 + 7^{-2 \mathfrak{r}} )\frac{\mathfrak{r}}{\mathfrak{r}+1} \frac{\| u \|_{\ell^1} }{2\rho^\sharp} \leq \frac{\| u \|_{\ell^1} }{2\rho^\sharp}<1.
\end{split}
\end{equation}

Then, we note that by composition it is clear that $\tau^{(1),\sharp}$ and $\tau^{(0),\sharp} $ are symplectic. Now, we aim at proving that $ \tau^{(0),\sharp}$  is close to the identity in $\ell^1$. Indeed, if $\|u\|_{\ell^1} < \rho^\sharp$, we have
\begin{equation*}
\begin{split}
 \| \tau^{(0),\sharp}(u) \|_{\ell^1} &\leq \left( 1 +  \Big( \frac{\| \tau^{(0)}(u) \|_{\ell^1}}{\varepsilon_\chi} \Big)^{2\mathfrak{r}} \right) \| \tau^{(0)}(u) \|_{\ell^1} \leq \left( 1 +  \Big( \frac{\| u \|_{\ell^1}}{7 \rho} \Big)^{2\mathfrak{r}} \right) \| \tau^{(0)}(u) \|_{\ell^1}  \\
&\leq \left( 1 +  \Big( \frac{\| u \|_{\ell^1}}{7 \rho} \Big)^{2\mathfrak{r}} \right)  \left( 1 +  \Big( \frac{\| u \|_{\ell^1}}{\rho} \Big)^{2} \right)  \| u \|_{\ell^1}   \leq \left( 1 +   \Big( \frac{\| u \|_{\ell^1}}{\rho} \Big)^{2} + 2 \Big( \frac{\| u \|_{\ell^1}}{7 \rho} \Big)^{2\mathfrak{r}} \right)    \| u \|_{\ell^1} \\
&\leq \left( 1 + (1+ 2\, 7^{-2 \mathfrak{r} })   \Big( \frac{\mathfrak{r} }{\mathfrak{r}+1} \Big)^{2}     \Big( \frac{\| u \|_{\ell^1}}{\rho^\sharp} \Big)^{2}  \right)    \| u \|_{\ell^1} \leq  \left( 1 +    \Big( \frac{\| u \|_{\ell^1}}{\rho^\sharp} \Big)^{2}  \right)    \| u \|_{\ell^1}.
\end{split} 
\end{equation*}
We note that as a consequence $ \| \tau^{(0),\sharp}(u) \|_{\ell^1}  < 2  \rho^\sharp$ and so $ \| \tau^{(0),\sharp}(u) \|_{\ell^1} < \varepsilon_\chi$. Therefore as a consequence of the induction hypothesis and the invertibility of Hamiltonian flow of $\chi$ (i.e. \eqref{eq:invertPhichi}) the diagram \eqref{mon_beau_diagram_roi_des_forets} commute at the step $\mathfrak{r} +1$.

Now, we aim at proving that $ \tau^{(1),\sharp}$  is close to the identity in $\ell^1$. Indeed, if $\|u\|_{\ell^1} < 2\rho^\sharp$, we have
\begin{equation*}
\begin{split}
 \| \tau^{(1),\sharp}(u) \|_{\ell^1} &\leq \left( 1 +  \Big( \frac{\| \Phi_\chi^1 (u) \|_{\ell^1}}{2\rho} \Big)^{2} \right) \| \Phi_\chi^1 (u) \|_{\ell^1} \\
  &\mathop{\leq}^{\eqref{eq:cestlesvacances}} \left( 1 +   \Big[(1 + 7^{-2 \mathfrak{r}} )\frac{\mathfrak{r}}{\mathfrak{r}+1}\Big]^2  \Big( \frac{\| u \|_{\ell^1}}{2\rho^\sharp} \Big)^{2} \right) \| \Phi_\chi^1 (u) \|_{\ell^1} \\
 &\leq \left( 1 +   \Big[(1 + 7^{-2 \mathfrak{r}} )\frac{\mathfrak{r}}{\mathfrak{r}+1}\Big]^2  \Big( \frac{\| u \|_{\ell^1}}{2\rho^\sharp} \Big)^{2} \right)  \left( 1 +    7^{-2 \mathfrak{r}}    \Big( \frac{\| u \|_{\ell^1}}{2\rho^\sharp} \Big)^{2} \right)
  \| u \|_{\ell^1} \\
  &\leq \left[ 1 +  \Big( \Big[(1 + 7^{-2 \mathfrak{r}} )\frac{\mathfrak{r}}{\mathfrak{r}+1}\Big]^2 + 2 \, 7^{-2 \mathfrak{r}} \Big)    \Big( \frac{\| u \|_{\ell^1}}{2\rho^\sharp} \Big)^{2} \right] \| u \|_{\ell^1}\leq  \left( 1 +    \Big( \frac{\| u \|_{\ell^1}}{2\rho^\sharp} \Big)^{2} \right)\| u\|_{\ell^1}.
\end{split} 
\end{equation*}
Finally, the $H^s$ estimate \eqref{eq:growthHs} of $\tau^{(1),\sharp}$ and $\tau^{(0),\sharp}$ is a direct corollary of their definition and of the $H^s$ estimate of $\Phi_\chi^t$ (see \eqref{eq:est_tamevf}).

\medskip

\noindent \emph{\underline{$\triangleright$ Step 4 : The new expansion.}} 
Let us note that since $\Phi_\chi$ is the Hamiltonian flow of $\chi$ (see \eqref{eq:cestleflot}), if $Q$ is a smooth real valued function on $\ell^1$, provided that $\|u\|_{\ell^1}<2\rho^\sharp$ and $|t|\leq 1$, we have
$$
\frac{\mathrm{d}}{\mathrm{d}t} Q(\Phi_\chi^t(u)) = \{\chi , Q \}(\Phi_\chi^t(u)).
$$
Therefore, since $\Phi_\chi^0 = \mathrm{id}_{\ell^1}$, the Taylor expansion of $Q(\Phi_\chi^t(u))$ in $t=0$ at the order $m$ is given by
$$
Q(\Phi_\chi^1(u)) = \sum_{n=0}^{m} \frac{\mathrm{ad}_\chi^n Q(u)}{n!} + \int_0^1 \frac{(1-t)^m}{m!} \, \mathrm{ad}_\chi^{m+1} Q(u) \mathrm{d}t.
$$
As a consequence, we also get explicitly the Taylor expansion of $L^{(2q)} \circ \Phi_\chi^1$ at the order $2r$ : 
$$
L^{(2q)} \circ \Phi_\chi^1(u)=  \sum_{n=0}^{m_q} \frac{\mathrm{ad}_\chi^n L^{(2q)} (u)}{n!} + \int_0^1 \frac{(1-t)^{m_q}}{m_q!} \, \mathrm{ad}_\chi^{m_q+1} L^{(2q)} (u) \mathrm{d}t.
$$
where $m_q$ is the smallest index such that $(m_q+1) \mathfrak{r} + q > r $. Now, recalling that by induction hypothesis, we have $H \circ \tau^{(1)} = Z_2+ L^{(4)} + \cdots + L^{(2r)}+ R$ and that, by construction, $\tau^{(1),\sharp} = \tau^{(1)} \circ \Phi_\chi^1$, we deduce that
$$
H \circ \tau^{(1),\sharp} = Z_2+ L^{(4),\sharp} + \cdots + L^{(2r),\sharp}+ R^{\sharp}
$$
where 
$$
L^{(2q),\sharp} := \sum_{ \substack{ n \geq 0, \ q'\geq 2 \\
n \mathfrak{r} + q' = q}} \frac1{n!} \mathrm{ad}_\chi^n L^{(2q')} +  \sum_{ \substack{ n \geq 1 \\ n \mathfrak{r} + 1 = q}} \frac1{n!} \mathrm{ad}_\chi^{n-1}  \{\chi,Z_2 \}.
$$
and
\begin{equation}
\label{eq:defRem}
R^{\sharp} = \sum_{q=2}^r \int_0^1 \frac{(1-t)^{m_q}}{m_q!} \, \mathrm{ad}_\chi^{m_q+1} L^{(2q)} \circ \Phi_\chi^t  \mathrm{d}t + \int_0^1 \frac{(1-t)^{m_1}}{m_1!} \, \mathrm{ad}_\chi^{m_1} \{\chi,Z_2 \}  \circ \Phi_\chi^t  \mathrm{d}t + R\circ  \Phi_\chi^1.
\end{equation}
We are going to estimate this new remainder term carefully at the next steps of the proof. For the moment, we focus on estimating $L^{(2q),\sharp}$.

We recall that $-\{\chi,Z_2 \} = L^{(2\mathfrak{r}+2)} - L^{(2\mathfrak{r}+2),\sharp} \in \mathscr{H}_{2\mathfrak{r}+2}$ is nothing but the $\nu$-nonresonant part of $L^{(2\mathfrak{r}+2)}$ (see \eqref{eq:cestlapartienonresonnantedemachin}). Therefore, we have
\begin{equation}
\label{eq:jolie_astuce}
\big\| \frac1{n!} L^{(2\mathfrak{r}+2)}  + \frac1{(n+1)!} \{\chi,Z_2 \} \big\|_{\ell^\infty} \leq \frac1{n!} \| L^{(2\mathfrak{r}+2)} \|_{\ell^\infty}
\end{equation}
and so we do not have to take into account the contribution of the terms associated with $Z_2$ in the estimate of $\| L^{(2q),\sharp} \|_{\ell^\infty}$ for $q\geq \mathfrak{r}+1$.
As a consequence, we deduce of Proposition \ref{prop:Poisson} that, for all $2\leq q\leq r$, $L^{(2q),\sharp} \in \mathscr{H}_{2q}$ is indeed a homogeneous polynomial of degree $2q$ and that it satisfies the bound
\begin{equation}
\label{eq:pouet_pouet}
\| L^{(2q),\sharp} \|_{\ell^\infty} \leq \sum_{ \substack{ n \geq 0, \ q'\geq 2 \\
n \mathfrak{r} + q' = q}} \frac{4^n (\mathfrak{r}+1)^n \| L^{(2\mathfrak{r}+2)} \|_{\ell^\infty}^n }{\nu^n n!}   \| L^{(2q')} \|_{\ell^\infty}  \prod_{j=0}^{n-1} (q'+ \mathfrak{r} j).
\end{equation}
By induction hypothesis, we know that $\| L^{(2q)} \|_{\ell^\infty} \leq C_2^{2q - 3} \nu^{-q+2} \min(q,\mathfrak{r})^{2(q-2)} $ for all $q$. We aim at proving that $\| L^{(2q,\sharp)} \|_{\ell^\infty} \leq C_2^{2q - 3} \nu^{-q+2} \min(q,\mathfrak{r}+1)^{2(q-2)}$.

First, we note that since the sum in \eqref{eq:pouet_pouet} is trivial for $q\leq \mathfrak{r}+1$ (i.e. it is reduced to $n=0$), we only have to focus on the case $q \geq \mathfrak{r}+2$ (else the estimate is obvious).

Then, using the induction hypothesis in \eqref{eq:pouet_pouet}, for $q \geq \mathfrak{r}+2$, we get\footnote{we simply control the product by $q^n$.}
\begin{multline*}
B_q^\sharp := \frac{\| L^{(2q),\sharp} \|_{\ell^\infty} }{C_2^{2q - 3} \nu^{-q+2} (\mathfrak{r}+1)^{2(q-2)}}\leq \sum_{ \substack{ n \geq 0, \ q'\geq 2 \\
n \mathfrak{r} + q' = q}}  \nu^{-n+(q-2) -(q'-2) - n(\mathfrak{r}-1 ) }  C_2^{  -(2q-3) + n( 2 \mathfrak{r}-1 ) +  (2q'-3) }  \\
\times (\mathfrak{r}+1)^{n-2(q-2)}  \mathfrak{r}^{ 2n(\mathfrak{r}-1) + 2(q'-2) }   \frac{4^n q^n}{n !} 
\end{multline*}
Therefore, thanks to the relation $n \mathfrak{r} + q' = q$, we get
$$
B_q^\sharp \leq  \sum_{ \substack{ n \geq 0, \ q'\geq 2 \\ n \mathfrak{r} + q' = q}} C_2^{-n}    (\mathfrak{r}+1)^{n-2(q-2)}  \mathfrak{r}^{ 2(q-2) -2n}      \frac{4^n q^n}{n !} \leq  \sum_{ \substack{ n \geq 0, \ q'\geq 2 \\ n \mathfrak{r} + q' = q}}  q^n \Big(  \frac{ \mathfrak{r}}{  \mathfrak{r}+1} \Big)^{2(q-2)}      \frac{\mathfrak{r}^{-n} 8^n C_2^{-n} }{n !}.
$$
Then, using the estimate $A^{-q} q^n \leq e^{-n} n^n (\log(A))^{-n} $ whenever $A>1$ and the convexity of the logarithm, we get
$$
q^n \Big(  \frac{ \mathfrak{r}}{  \mathfrak{r}+1} \Big)^{2(q-2)}  \leq 2^4 2^{-n} e^{-n} n^n \Big(\log\big( 1+ \frac1{\mathfrak{r}}  \big)\Big)^{-n} \leq  2^4  e^{-n} n^n \big(\log(2)\big)^{-n} \mathfrak{r}^n.
$$
Therefore, since $e^{-n} n^n   \leq  n! $, we get
$$
B_q^\sharp \leq \Big(  \frac{ \mathfrak{r}}{  \mathfrak{r}+1} \Big)^{2(q-2)}  + \sum_{ \substack{ n \geq 1, \ q'\geq 2 \\ n \mathfrak{r} + q' = q}}  2^4  8^{n}  \big(\log(2)\big)^{-n} C_2^{-n}.
$$
Moreover, since $q\geq\mathfrak{r}+2$ and $\mathfrak{r}\geq 1$, we have 
$$
 \Big(  \frac{ \mathfrak{r}}{  \mathfrak{r}+1} \Big)^{2(q-2)}   \leq  \Big(  \frac{ \mathfrak{r}}{  \mathfrak{r}+1} \Big)^{2\mathfrak{r}} \leq e^{-1} 
$$
and so, as expected, since $C_2$ can be chosen large enough,
\begin{equation}
\label{eq:C22}
B_q^\sharp \leq e^{-1} + \frac{2^8 }{\log(2)} C_2^{-1} \leq 1.
\end{equation}

\medskip

\noindent \emph{\underline{$\triangleright$ Step 5 : Estimate of the new remainder term in $\ell^1$.}} Finally, we just have to control the new remainder term $R^\sharp$ (given by \eqref{eq:defRem}).  Reasoning as in \eqref{eq:jolie_astuce}, with a small abuse of notations, we ignore the contribution of the terms coming from $\{\chi,Z_2\}$.

We have to control terms of the form  $\nabla (Q \circ \Phi^{t}_\chi )(u)$ where $\|u\|_{\ell^1} < 2\rho^{\sharp}$ and $Q$ is a smooth function in $\ell^1$. First, let us prove the following formula which is very convenient 
\begin{equation}
\label{eq:jaimepaslesadjoints}
\nabla (Q \circ \Phi^{t}_\chi )(u) = -i (\mathrm{d} \Phi^{-t}_\chi) (\Phi^{t}_\chi(u)) \big[i (\nabla Q) (\Phi^{t}_\chi(u))\big] \quad on \quad B_{\ell^1}(0,2 \rho^\sharp).
\end{equation}
Indeed, by \eqref{eq:invertPhichi} and \eqref{eq:closeidl1}, we have $\Phi_\chi^t \circ \Phi_\chi^{-t} = \mathrm{id}_{\ell^1}$
on $B_{\ell^1}(0,\varepsilon_\chi/2)$ and so 
$$
[(\mathrm{d}\Phi_\chi^t)  \circ \Phi_\chi^{-t}] \mathrm{d}\Phi_\chi^{-t} = \mathrm{id}_{\ell^1} \quad on \quad B_{\ell^1}(0,\varepsilon_\chi/2).
$$
We note that if $\| u \|_{\ell^1} < 2 \rho^\sharp$ then $\| \Phi_\chi^{t}(u) \|_{\ell^1} < 2\rho <\varepsilon_\chi/2$ (see \eqref{eq:valabon},\eqref{eq:cestlesvacances}) and so we have
$$
(\mathrm{d}\Phi_\chi^t)   [(\mathrm{d}\Phi_\chi^{-t}) \circ \Phi_\chi^t] = \mathrm{id}_{\ell^1} \quad on \quad B_{\ell^1}(0,2 \rho^\sharp).
$$
Therefore, since $\Phi_\chi^t$ is symplectic, if $\|u\|_{\ell^1} < 2\rho^{\sharp}$, we have, for all $v \in \ell^1$
\begin{equation*}
\begin{split}
&(\nabla (Q \circ \Phi_\chi^t)(u),v)_{L^2} = \mathrm{d} (Q \circ  \Phi_\chi^t)(u)(v) = ((\nabla Q) \circ \Phi_\chi^t (u),\mathrm{d}\Phi_\chi^t (u)(v))_{L^2} \\
=&  (i(\nabla Q) \circ \Phi_\chi^t (u),i\mathrm{d}\Phi_\chi^t (u)(v))_{L^2} =  \Big(\mathrm{d}\Phi_\chi^t(u)   [(\mathrm{d}\Phi_\chi^{-t})( \Phi_\chi^t(u))] \big(i(\nabla Q) \circ \Phi_\chi^t (u)\big),i\mathrm{d}\Phi_\chi^t (u)(v)\Big)_{L^2} \\
=& ( (\mathrm{d}\Phi_\chi^{-t})( \Phi_\chi^t(u)) \big(i(\nabla Q) \circ \Phi_\chi^t (u)\big),i v )_{L^2} = (-i (\mathrm{d}\Phi_\chi^{-t})( \Phi_\chi^t(u)) \big(i(\nabla Q) \circ \Phi_\chi^t (u)\big), v )_{L^2}
\end{split}
\end{equation*}
which is clearly equivalent to \eqref{eq:jaimepaslesadjoints}.

Now thanks to this identity \eqref{eq:jaimepaslesadjoints} and the estimate \eqref{eq:closeiddiffl1} on $\mathrm{d} \Phi^{-t}_\chi$, by the triangular inequality, we have
\begin{equation*}
\begin{split}
\|\nabla R^\sharp(u) \|_{\ell^1} &\leq \|\nabla (R \circ \Phi_\chi^1)(u) \|_{\ell^1} + \max_{0\leq t \leq 1} \sum_{q=2}^r  \frac{1}{m_q!} \, \| \nabla (\mathrm{ad}_\chi^{m_q+1} L^{(2q)} \circ \Phi_\chi^t (u))\|_{\ell^1} \\
&\lesssim \|(\nabla R) \circ \Phi_\chi^1(u) \|_{\ell^1} + \max_{0\leq t \leq 1} \sum_{q=2}^r  \frac{1}{m_q!} \, \| \nabla (\mathrm{ad}_\chi^{m_q+1} L^{(2q)}) \circ \Phi_\chi^t (u)\|_{\ell^1}.
\end{split}
\end{equation*}
Then, estimating the vector field of $\mathrm{ad}_\chi^{m_q+1} L^{(2q)}$ by Corollary \ref{cor:est_vfl1}, we get
\begin{equation*}
\begin{split}
\|\nabla R^\sharp(u) \|_{H^s} \lesssim   r \sum_{q=2}^r  \frac{ \| \mathrm{ad}_\chi^{m_q+1} L^{(2q)} \|_{\ell^\infty}}{m_q!} \,  \|  \Phi_\chi^t (u)\|_{L^\infty_t \ell^1}^{2 f_q+1}
+ \|(\nabla R) \circ \Phi_\chi^1(u) \|_{\ell^1}
\end{split}
\end{equation*}
where $f_q := (m_q+1) \mathfrak{r} + q-1$. We note that by definition of $m_q$, we have $r \leq f_q < 2r$. Therefore, since $\Phi_\chi^t$ is close to the identity, we have the estimates 
$$ \|  \Phi_\chi^t (u)\|_{L^\infty_t \ell^1}^{2 f_q+1} \leq 2^{4r+1} \|u\|_{\ell^1}^{2 f_q+1}
$$
Estimating the sum as at the previous step (thanks to Proposition \ref{prop:Poisson}), it comes
\begin{equation*}
\begin{split}
\|\nabla R^\sharp(u) \|_{\ell^1} \lesssim & 2^{5r} \sum_{q=2}^r  \big( \frac{8}{(\log 2) C_2} \big)^{m_q+1}   (m_q+1) \big( C_2^{2f_q  -1} \nu^{-f_q+1} (\mathfrak{r}+1)^{2(f_q-1)} \big)  \|  u\|_{\ell^1}^{2 f_q+1} 
\\ &+ \|(\nabla R) \circ \Phi_\chi^1(u) \|_{\ell^1}.
\end{split}
\end{equation*}
Recalling that by definition $(\mathfrak{r}+1)^{2} \nu^{-1} C_2^{2} = (16 \cdot 7  \rho^{\sharp})^2 \geq (2\rho^{\sharp} )^2$, we have
\begin{equation*}
\begin{split}
 C_2^{2f_q - 1} \nu^{-f_q+1} (\mathfrak{r}+1)^{2(f_q-1)}   \|  u\|_{ \ell^1}^{2 f_q} &\leq C_2^{ 2r - 1} \nu^{-r+1} (\mathfrak{r}+1)^{2(r-1)} \big( \frac{\|  u\|_{ \ell^1} }{ 2 \rho^\sharp}\big)^{2 (f_q -r)} \|  u\|_{ \ell^1}^{2r} \\
 &\leq C_2 (C_2^2 \nu^{-1} r^{2})^{r-1}  \|  u\|_{ \ell^1}^{2r}
 \end{split}
\end{equation*}
and so (since $m_q \leq r$ and $ \frac{8}{(\log 2) C_2}<1$)
$$
\|\nabla R^\sharp(u) \|_{\ell^1} \lesssim  (2^5 C_2^2 \nu^{-1} r^{2})^{r-1}  \|  u\|_{ \ell^1}^{2r+1}  +  \|(\nabla R) \circ \Phi_\chi^1(u) \|_{\ell^1}.
$$
Now, thanks to induction hypothesis, since $ \Phi_\chi^1$ is close to the identity, we have
\begin{equation*}
\begin{split}
 \|(\nabla R) \circ \Phi_\chi^1(u) \|_{\ell^1} &\leq   K_{\ell^1}^{\mathfrak{r}} (2^5 C_2^2 \nu^{-1} r^{2})^{r-1} \big( \prod_{j=1}^{\mathfrak{r}-1} (1+2^{-2j} ) \big)^{2r}  \|  \Phi_\chi^1(u) \|_{ \ell^1}^{2r+1}  \\
 &\leq  K_{\ell^1}^{\mathfrak{r}} (2^5 C_2^2 \nu^{-1} r^{2})^{r-1} \big( \prod_{j=1}^{\mathfrak{r}} (1+2^{-2j} ) \big)^{2r}  \|  u \|_{ \ell^1}^{2r+1} 
 \end{split}
\end{equation*}
and so
\begin{equation}
\label{eq:defKell1}
\|\nabla R^\sharp(u) \|_{\ell^1} \lesssim (2^5 C_2^2 \nu^{-1} r^{2})^{r-1}  (1  +   K_{\ell^1}^{\mathfrak{r}}  \big( \prod_{j=1}^{\mathfrak{r}} (1+2^{-2j} ) \big)^{2r}  )\|  u \|_{ \ell^1}^{2r+1}.
\end{equation}
Therefore, provided that the constant $K_{\ell^1}>1$ is chosen large enough, we deduce that as expected 
$$
\|\nabla R^\sharp(u) \|_{\ell^1} \leq K_{\ell^1}^{\mathfrak{r}+1} (2^5 C_2^2 \nu^{-1} r^{2})^{r-1} \big( \prod_{j=1}^{\mathfrak{r}} (1+2^{-2j} ) \big)^{2r}  \|  u\|_{ \ell^1}^{2r+1}.
$$

\medskip

\noindent \emph{\underline{$\triangleright$ Step 6 : Estimate of the new remainder term in $H^s$.}} The proof is similar to the one in $\ell^1$. As previously, with a small abuse of notations, we ignore the contribution of the terms coming from $\{\chi,Z_2\}$. Thanks to this identity \eqref{eq:jaimepaslesadjoints} and the estimate \eqref{eq:est_tamedvf} on $\mathrm{d} \Phi^{-t}_\chi$, by the triangular inequality, we have
\begin{equation*}
\begin{split}
\|\nabla R^\sharp(u) \|_{H^s} \leq& \|\nabla (R \circ \Phi_\chi^1)(u) \|_{H^s} + \max_{0\leq t \leq 1} \sum_{q=2}^r  \frac{1}{m_q!} \, \| \nabla (\mathrm{ad}_\chi^{m_q+1} L^{(2q)} \circ \Phi_\chi^t (u))\|_{H^s} \\
\lesssim_s& \|(\nabla R) \circ \Phi_\chi^1(u) \|_{H^s} +  \sum_{q=2}^r  \frac{1}{m_q!} \, \| \nabla (\mathrm{ad}_\chi^{m_q+1} L^{(2q)}) \circ \Phi_\chi^t (u)\|_{L^\infty_t H^s} \\
&+ \varepsilon_\chi^{-1} \| \Phi_\chi^t (u)\|_{L^\infty_t H^s} \Big( \|(\nabla R) \circ \Phi_\chi^1(u) \|_{\ell^1} +  \sum_{q=2}^r  \frac{1}{m_q!} \, \| \nabla (\mathrm{ad}_\chi^{m_q+1} L^{(2q)}) \circ \Phi_\chi^t (u)\|_{L^\infty_t \ell^1} \Big).
\end{split}
\end{equation*}
Fortunately, the previous step has been devoted to estimating this last term in parenthesis. Therefore, since $\|  \Phi_\chi^t (u)\|_{L^\infty_t H^s} \lesssim_s \|u\|_{H^s}$ and $\|u\|_{\ell^1} \varepsilon_\chi^{-1}<1$, we have
\begin{equation*}
\begin{split}
\|\nabla R^\sharp(u) \|_{H^s} \lesssim_s& \|(\nabla R) \circ \Phi_\chi^1(u) \|_{H^s} +  \sum_{q=2}^r  \frac{1}{m_q!} \, \| \nabla (\mathrm{ad}_\chi^{m_q+1} L^{(2q)}) \circ \Phi_\chi^t (u)\|_{L^\infty_t H^s} \\
&+ K_{\ell^1}^{\mathfrak{r}} (2^5 C_2^2 \nu^{-1} r^{2})^{r-1} \big( \prod_{j=1}^{\mathfrak{r}} (1+2^{-2j} ) \big)^{2r}  \|  u\|_{ \ell^1}^{2r}\|  u\|_{ H^s}  .
\end{split}
\end{equation*}
Estimating the vector field of $\mathrm{ad}_\chi^{m_q+1} L^{(2q)}$ by Corollary \ref{cor:est_vf} and proceeding as previously (i.e. as in step $5$), we get
$$
 \sum_{q=2}^r  \frac{1}{m_q!} \, \| \nabla (\mathrm{ad}_\chi^{m_q+1} L^{(2q)}) \circ \Phi_\chi^t (u)\|_{L^\infty_t H^s} \lesssim_s  (2^5 C_2^2 \nu^{-1} r^{2})^{r-1}  \|  u\|_{ \ell^1}^{2r} \|  u\|_{ H^s} 
$$
Now, thanks to induction hypothesis, since $ \Phi_\chi^1$ is close to the identity (in $\ell^1$ and $H^s$) we have
\begin{equation*}
\begin{split}
 \|(\nabla R) \circ \Phi_\chi^1(u) \|_{H^s} &\leq  K_s^{\mathfrak{r}} (2^5 C_2^2 \nu^{-1} r^{2})^{r-1} \big( \prod_{j=1}^{\mathfrak{r}-1} (1+2^{-2j} ) \big)^{2r}  \|  \Phi_\chi^1(u) \|_{ \ell^1}^{2r} \|  \Phi_\chi^1(u) \|_{ H^s} \\
 &\lesssim_s   K_s^{\mathfrak{r}} (2^5 C_2^2 \nu^{-1} r^{2})^{r-1} \big( \prod_{j=1}^{\mathfrak{r}} (1+2^{-2j} ) \big)^{2r}  \|  u \|_{ \ell^1}^{2r} \|  u \|_{ H^s} 
 \end{split}
\end{equation*}
and so
\begin{equation}
\label{eq:defK_s}
\|\nabla R^\sharp(u) \|_{H^s} \lesssim_s (2^5 C_2^2 \nu^{-1} r^{2})^{r-1}  (1  +  (K_s^{\mathfrak{r}} +  K_{\ell^1}^{\mathfrak{r}}) \big( \prod_{j=1}^{\mathfrak{r}} (1+2^{-2j} ) \big)^{2r}  )\|  u \|_{ \ell^1}^{2r} \|  u \|_{ H^s}.
\end{equation}
Therefore, the constant $K_s\geq K_{\ell^1} >1$ can be chosen large enough\footnote{note it can be easily checked that this definition is not circular : the constant in \eqref{eq:defK_s} does not depend on $K_s$.} to get
$$
\|\nabla R^\sharp(u) \|_{H^s} \leq K_s^{\mathfrak{r}+1} (2^5 C_2^2 \nu^{-1} r^{2})^{r-1} \big( \prod_{j=1}^{\mathfrak{r}} (1+2^{-2j} ) \big)^{2r}  \|  u\|_{ \ell^1}^{2r} \| u \|_{ H^s} 
$$
which conclude the induction.

\end{proof}

\section{Proof of Theorem \ref{th-main}}

In this section we aim at proving Theorem \ref{th-main} thanks to the Birkhoff normal form theorem (Thm \ref{thm:Birk}). Therefore, we recall that \eqref{NLS} rewrites as an Hamiltonian system
$$
i\partial_t u = \frac12 \nabla H(u)
$$
where the Hamiltonian $H$ of \eqref{NLS} writes
\begin{equation}
\label{eq:HamiltofNLS}
H(u) := Z_2(u) + P(u)  
\end{equation}
with
$$
Z_2(u) = \int_{\mathbb{T}^d} |\nabla u(x)|^2 + (V \ast u)(x) \overline u(x)  \mathrm{d}x \quad \mathrm{and} \quad  P(u) = \int_{\mathbb{T}^d} \frac{\sigma}{p+1} |u(x)|^{2p+2} \mathrm{d}x.
$$
In order to apply the results of Section \ref{sec:Hamform}, it is worth to notice that these functions rewrites
\begin{equation}
\label{eq:defZ2NLS}
Z_2(u) =  \sum_{k \in \mathbb{Z}^d} \omega_k |u_k|^2 \quad \mathrm{where} \quad \omega_k := |k|^2 + (2\pi)^{-d/2} V_k
\end{equation}
and 
$$
P(u) = \frac{\sigma (2\pi)^{-pd}}{p+1} \sum_{\boldsymbol{k}_1 + \cdots + \boldsymbol{k}_{p+1} = \boldsymbol{\ell}_1 + \cdots + \boldsymbol{\ell}_{p+1} } u_{\boldsymbol{k}_1} \dots u_{\boldsymbol{k}_{p+1}} \overline{u_{\boldsymbol{\ell}_1}} \dots \overline{u_{\boldsymbol{\ell}_{p+1}}}.
$$
Note that $P\in \mathscr{H}_{2p+2}$ and satisfies $\| P\|_{\ell^\infty} \leq (2\pi)^{-pd} (p+1)^{-1}$.

Before proving Theorem \ref{th-main} thanks to a bootstrap argument, we begin with two technical subsections in which we define the non-resonant potential (i.e. the set $\mathcal{V}$) and we study the resonant Hamiltonians (according to Definition \ref{def:resHam}).
\subsection{Non resonant potentials}\label{sec:NR} First, we define the set $\mathcal{V}$ of the Fourier multiplier $V$ for which we are going to prove Theorem \ref{th-main}.

\begin{definition}[Set $\mathcal{V}$] \label{def:setV} A Fourier multiplier $V \in \ell^\infty(\mathbb{Z}^d; \mathbb{R})$ belongs to $\mathcal{V}$ if exists $\gamma>0$ such that for all $q\geq 2$, all $\boldsymbol{k},\boldsymbol{\ell} \in (\mathbb{Z}^d)^q$ we have
$$
|\Omega(\boldsymbol{k},\boldsymbol{\ell})| \geq \gamma\, q^{-4} \big(\log_2 \max_{1\leq j\leq q} (| \boldsymbol{k}_j |, | \boldsymbol{\ell}_j |)\big)^{-(2q+1)} \quad \mathrm{whenever} \quad (\boldsymbol{k}, \boldsymbol{\ell}) \quad \mathrm{satisfies \ \eqref{eq:monomials_to_remove}.}
$$
\end{definition} 
\begin{remark} We recall that the small divisors $\Omega(\boldsymbol{k},\boldsymbol{\ell})$ are defined in Definition \ref{def:SD} (the frequencies $\omega_k$ being given by \eqref{eq:defZ2NLS}) and  that the condition "$(\boldsymbol{k}, \boldsymbol{\ell})$ satisfies  \eqref{eq:monomials_to_remove}" only means that the monomial $u_{\boldsymbol{k}_1} \dots u_{\boldsymbol{k}_q} \overline{u_{\boldsymbol{\ell}_1}} \dots \overline{u_{\boldsymbol{\ell}_q}}$ does not commute with the super-actions $J_n$ (defined in \eqref{eq:def_Ns_J_n}).
\end{remark}

\begin{remark} In order to include more potentials, we could easily extend this definition by considering estimates of the form
$
|\Omega(\boldsymbol{k},\boldsymbol{\ell})| \geq \gamma\, c_1^q \big(\log \max_{1\leq j\leq q} (| \boldsymbol{k}_j |, | \boldsymbol{\ell}_j |)\big)^{-c_2 q}
$
where $c_1,c_2>0$ would depend on $V$ but for simplicity we chose to have explicit constants (the constant $4$ in the definition of $T_\varepsilon$ in Theorem \ref{th-main} would then be related to $c_2$). 
\end{remark}

In the following lemma (which is the main result of this section), we prove that the set $\mathcal{V}$ is non-empty.
\begin{lemma}\label{lemma:nonres}
Almost surely, the random potential $V$ defined by \eqref{eq:def_V_rand} belongs to $\mathcal{V}$. 
\end{lemma}
\begin{proof}

We aim at estimating the probability of the following events
$$
E_\gamma := \big\{\, \exists q\geq 2, \exists \boldsymbol{k},\boldsymbol{\ell} \in (\mathbb{Z}^d)^q,\quad (\boldsymbol{k}, \boldsymbol{\ell}) \quad \mathrm{satisfies \ \eqref{eq:monomials_to_remove}}  \quad \Rightarrow \quad |\Omega(\boldsymbol{k},\boldsymbol{\ell})|\leq \gamma c_{\boldsymbol{k},\boldsymbol{\ell}} \, \big\}
$$
where $\gamma>0$ and the constants $c_{\boldsymbol{k},\boldsymbol{\ell}}>0$ will be defined later. \\
\emph{\underline{Step $1$ : To Make the multiplicities appear.}} By definition of $V$, we note that $\Omega(\boldsymbol{k},\boldsymbol{\ell})$ writes under the form 
$$
|\Omega(\boldsymbol{k},\boldsymbol{\ell})| = 2 \, | a + (2\pi)^{-d/2} (X_{\boldsymbol{n}_1}+ \cdots + X_{\boldsymbol{n}_q} - X_{\boldsymbol{m}_1} - \cdots -X_{\boldsymbol{m}_q}  )  | =: \Upsilon_a(\boldsymbol{n},\boldsymbol{m})
$$
where $a\in \mathbb{Z}$ is an integer (depending on $\boldsymbol{k},\boldsymbol{\ell}$) and $\boldsymbol{n}_j$ is the integer such that $\boldsymbol{k}_j \in B_{\boldsymbol{n}_j}$ ($\boldsymbol{m}_j$ being defined similarly with respect to $\boldsymbol{\ell}_j$) and $(B_{n})_n$ denotes the usual dyadic decomposition of the Fourier space given by \eqref{eq:def_dya}. Therefore, provided that $\rho_{\boldsymbol{n},\boldsymbol{m}}>0$ is a constant such that 
$\rho_{\boldsymbol{n},\boldsymbol{m}} \geq c_{\boldsymbol{k},\boldsymbol{\ell}},$ by definition of the non-resonance condition \eqref{eq:monomials_to_remove},
we have
$$
E_\gamma \subset  \bigcup_{q \geq 2} \bigcup_{ \substack{  \boldsymbol{n},\boldsymbol{m} \in \mathbb{N}^q \\ \boldsymbol{n}  \notin \mathfrak{S}_q \boldsymbol{m}  }} \bigcup_{a\in \mathbb{Z}} \big\{ \Upsilon_a(\boldsymbol{n},\boldsymbol{m}) \leq \gamma \rho_{\boldsymbol{n},\boldsymbol{m}} \, \big\}
$$
where $\boldsymbol{n}  \notin \mathfrak{S}_q \boldsymbol{m} $ just mean that $\boldsymbol{n}\neq \boldsymbol{m}$ up  to  a  permutation. As a consequence, we deduce the estimate
$$
\mathbb{P} (E_\gamma ) \leq \sum_{q \geq 2} \sum_{ \substack{  \boldsymbol{n},\boldsymbol{m} \in \mathbb{N}^q \\ \boldsymbol{n}  \notin \mathfrak{S}_q \boldsymbol{m}  }} \sum_{a\in \mathbb{Z}} \ \mathbb{P} \big( \Upsilon_a(\boldsymbol{n},\boldsymbol{m}) \leq \gamma \rho_{\boldsymbol{n},\boldsymbol{m}} \, \big).
$$
\emph{\underline{Step $2$ : To reduce the sum with respect to $a$.}} Now, we note that since the random variables $X_{n}$ are bounded by $1$, if $|a| > q$ then
$$
\Upsilon_a(\boldsymbol{n},\boldsymbol{m}) \geq 2 |a| - 4  (2\pi)^{-d/2} q \geq 2 (|a| - q) \geq 2.
$$
Therefore, assuming from now that $\gamma$ and $\rho_{\boldsymbol{n},\boldsymbol{m}}$ are such that $\gamma \rho_{\boldsymbol{n},\boldsymbol{m}}< 2$, we have
$$
|a|>q \quad \Rightarrow \quad \mathbb{P} \big( \Upsilon_a(\boldsymbol{n},\boldsymbol{m}) \leq \gamma \rho_{\boldsymbol{n},\boldsymbol{m}} \, \big) = 0
$$
and so
\begin{equation}
\label{eq:sieste}
\mathbb{P} (E_\gamma ) \leq \sum_{q \geq 2} (2q+1) \sum_{ \substack{  \boldsymbol{n},\boldsymbol{m} \in \mathbb{N}^q \\ \boldsymbol{n}  \notin \mathfrak{S}_q \boldsymbol{m}  }} \ \sup_{a\in \mathbb{Z}} \ \mathbb{P} \big( \Upsilon_a(\boldsymbol{n},\boldsymbol{m}) \leq \gamma \rho_{\boldsymbol{n},\boldsymbol{m}} \, \big).
\end{equation}
\emph{\underline{Step $3$ : Estimation of $\mathbb{P} \big( \Upsilon_a(\boldsymbol{n},\boldsymbol{m}) \leq \gamma \rho_{\boldsymbol{n},\boldsymbol{m}} \, \big)$.}}
We note that since $ \boldsymbol{n}  \notin \mathfrak{S}_q \boldsymbol{m}$,  $ \Upsilon_a(\boldsymbol{n},\boldsymbol{m}) $ writes under the form
$$
\Upsilon_a(\boldsymbol{n},\boldsymbol{m}) = 2 (2\pi)^{-d/2}| b X_k + Y |
$$
where $b\in \mathbb{Z}^*$, $k \in \{  \boldsymbol{n}_1,\cdots, \boldsymbol{n}_q ,  \boldsymbol{m}_1,\cdots,  \boldsymbol{m}_q  \}$ and $Y$ is random variable independent of $X_k$. Therefore, since $X_k \sim \mathcal{U}(0,1)$ is uniformly distributed in $(0,1)$, we have
\begin{equation}
\label{eq:brique}
\mathbb{P}(|\Omega(\boldsymbol{k},\boldsymbol{\ell})|\leq \gamma \rho_{\boldsymbol{n},\boldsymbol{m}}) = \mathbb{E} \int_0^1 \mathbbm{1}_{2 (2\pi)^{-d/2}|bx_k + Y|\leq \gamma \rho_{\boldsymbol{n},\boldsymbol{m}}}  \mathrm{d}x_k\leq (2\pi)^{d/2} \gamma \rho_{\boldsymbol{n},\boldsymbol{m}}.
\end{equation}
\emph{\underline{Step $4$ : Conclusion.}} Putting \eqref{eq:sieste} and \eqref{eq:brique} together we deduce that
\begin{equation}
\label{jeanaiunpeumarre}
\mathbb{P} (E_\gamma ) \lesssim_d \gamma  \sum_{q \geq 2} (2q+1) \sum_{ (\boldsymbol{n},\boldsymbol{m}) \in \mathbb{N}^{2q}\setminus \{(0,0)\} } \rho_{\boldsymbol{n},\boldsymbol{m}}.
\end{equation}
Therefore, we set
$$
c_{\boldsymbol{k},\boldsymbol{\ell}} =  q^{-4} \big(\log_2 \max_{1\leq j\leq q} (| \boldsymbol{k}_j |, | \boldsymbol{\ell}_j |)\big)^{-(2q+1)}  \quad \mathrm{and} \quad    \rho_{\boldsymbol{n},\boldsymbol{m}} =q^{-4}  (\max_{1\leq j\leq q} (| \boldsymbol{m}_j | , | \boldsymbol{n}_j | ))^{-(2q+1)}.
$$
Since whenever\footnote{ and so $\max_{1\leq j\leq q} (| \boldsymbol{k}_j |, | \boldsymbol{\ell}_j | ) \neq 0$ .} $\boldsymbol{n}  \notin \mathfrak{S}_q \boldsymbol{m}$ 
$$
\max_{1\leq j\leq q} (| \boldsymbol{k}_j |, | \boldsymbol{\ell}_j | ) \geq 2^{\max_{1\leq j\leq q} (| \boldsymbol{m}_j |, | \boldsymbol{n}_j |)}
$$
as required, we have $\rho_{\boldsymbol{n},\boldsymbol{m}} \geq c_{\boldsymbol{k},\boldsymbol{\ell}}$.
Finally, thank to \eqref{jeanaiunpeumarre} and the mean value inequality, we get 
$$
\mathbb{P} (E_\gamma )  \lesssim_d \gamma \sum_{q \geq 2} q^{-3} \sum_{m\geq 1} \frac{m^{2q} - (m-1)^{2q} }{m^{2q+1}} \lesssim_d \gamma \sum_{q \geq 2} q^{-2} \lesssim_d \gamma
$$
which is enough to deduce that $\mathbb{P} \big(\bigcap_{\gamma>0} E_\gamma \big) \leq \inf_{\gamma >0} \mathbb{P} (E_\gamma )  = 0$.

\end{proof}

\subsection{Smallness of the resonant Hamiltonian} \label{sec:loglosses}As we can see in our Birkhoff normal form theorem (Thm \ref{thm:Birk}), we do not have removed the $\nu$-resonant terms (associated with $L$ in Thm \ref{thm:Birk}). In this subsection, we are going to prove (in Proposition \ref{prop:maintech} below) that they do not make increase to much some observables $\mathcal{N}_{N,s}$ which are equivalent to the square of the $H^s$ norm.

\begin{definition}[$\mathcal{N}_{N,s}$]\label{def:NNs} Let $N \geq 1$ be an integer of the form $N=2^{n_{\max}}$ with $n_{\max} \in \mathbb{N}$. For all $s>0$ and $u\in H^s(\mathbb{T}^d)$, we set
$$
\mathcal{N}_{N,s}(u) = \mathcal{N}_{N,s}^{(low)}(u) + \mathcal{N}_{N,s}^{(high)}(u)
$$
where
$$
\mathcal{N}_{N,s}^{(low)} = \sum_{0\leq n < n_{\max}} (2^{n})^{2s} J_n \quad \mathrm{and} \quad \mathcal{N}_{N,s}^{(high)}(u) = \sum_{k\geq N} |k|^{2s} |u_k|^2.
$$
\end{definition}
We recall that the super actions $J_n$ are defined in \eqref{eq:def_Ns_J_n}. In the proof of Theorem \ref{th-main}, the parameter $N$ will be optimized with respect to $\varepsilon$ (the size of the initial datum). It will be chosen much larger than usually in the literature\footnote{usually the truncation parameter is of the form $N = \varepsilon^{- \eta}$ with $0<\eta \ll 1$ (see e.g. \cite{Bam03,BG06,KillBill}).} : it will be of the form $N= \varepsilon^{-r(\varepsilon)}$ where $r(\varepsilon)$ goes to $+\infty$ as $\varepsilon$ goes to $0$. Of course, as expected we note that these observables are equivalent to the square of the $H^s$ norm :
\begin{equation*}
2^{-2s} \| \cdot \|_{H^s}^2 \leq \mathcal{N}_{N,s} \leq \|  \cdot \|_{H^s}^2.
\end{equation*}

The following proposition is the main result of this section. We prove that the $\nu$-resonant Hamiltonians almost commute with the $\mathcal{N}_{N,s}$ norm.
\begin{proposition} \label{prop:maintech} Let $V\in \mathcal{V}$ be a non-resonant potential (and $\gamma>0$ be the associated constant),  $N \geq 1$ be an integer of the form $N=2^{n_{\max}}$ with $n_{\max} \in \mathbb{N}^*$, $q\geq 2$ be an integer and $\nu \in (0,1)$ be a small real number such that
\begin{equation}
\label{eq:une_CFL}
 \gamma q^{-4} \big(\log_2 \big( 2qN \big) \big)^{-(2q+1)} \geq \nu .
\end{equation}
If $L \in \mathscr{H}^{(\nu-\mathrm{res})}$ is a $\nu$-resonant homogeneous polynomial of degree $2q$ then for $s\geq0$, $\eta \in(0,1]$ and $u\in \ell^1_\eta \cap H^s(\mathbb{T}^d)$ we have
$$
|\{ \mathcal{N}_{N,s}, L \}(u)| \lesssim_{s} q^{2s+1}  N^{-\eta} \| L \|_{\ell^\infty}\| u\|_{\ell^1_\eta} \| u\|_{\ell^1}^{2q-3} \| u\|_{H^s}^{2}.
$$
\end{proposition}

The rest of this subsection is devoted to the proof of this proposition. In particular, from now we assume that $V \in \mathcal{V}$ and $L$ are fixed and that $\nu$, $q$ and $N$ satisfy the estimate \eqref{eq:une_CFL}. As usual, in order to prove our multi-linear estimates, we introduce the functions $\mu_1,\cdots,\mu_{2q} : (\mathbb{Z}^d)^{2q} \to \mathbb{R}_+$ such that for all $\boldsymbol{h} \in  (\mathbb{Z}^d)^{2q}$ and $j \in  \llbracket 1,2q\rrbracket$,  $\mu_j (\boldsymbol{h}) $ is the $j^{est}$ largest number among  $|\boldsymbol{h}_{1} |,\cdots, |\boldsymbol{h}_{2q} |$. In other words, $(\mu_j (\boldsymbol{h}))_j$ is the nondecreasing sequence which is equal to $(|\boldsymbol{h}_{j} |)_j$ up to a permutation\footnote{i.e. $\exists \sigma \in  \mathfrak{S}_{2q}, \forall j \in \llbracket 1 ,2q \rrbracket, \quad \mu_j (\boldsymbol{h}) = |\boldsymbol{h}_{\sigma_j} |$.}.

\begin{lemma}
\label{lem:outza}
If $\boldsymbol{k},\boldsymbol{\ell}\in (\mathbb{Z}^d)^q$ satisfy
$$
|\Omega(\boldsymbol{k},\boldsymbol{\ell})|< \nu \quad \mathrm{and} \quad \boldsymbol{k}_1+ \cdots + \boldsymbol{k}_q = \boldsymbol{\ell}_1+ \cdots + \boldsymbol{\ell}_q,
 $$
 then either $\mu_2(\boldsymbol{k},\boldsymbol{\ell}) \geq N$ or $(\boldsymbol{k}, \boldsymbol{\ell})$ does not satisfy  \eqref{eq:monomials_to_remove}.
\end{lemma}
\begin{proof}
We assume that $(\boldsymbol{k}, \boldsymbol{\ell})$ satisfies  \eqref{eq:monomials_to_remove}. Since $V\in \mathcal{V}$ is non-resonant and $\nu$, $q$ and $N$ satisfy the estimate \eqref{eq:une_CFL}, we have
$$
 \gamma q^{-4} \big(\log_2 \mu_1 (\boldsymbol{k}, \boldsymbol{\ell})\big)^{-(2q+1)} \leq \Omega(\boldsymbol{k},\boldsymbol{\ell}) \leq \nu \leq  \gamma q^{-4} \big(\log_2 \big( 2qN \big) \big)^{-(2q+1)}.
$$
As a consequence, we deduce that $\mu_1 (\boldsymbol{k}, \boldsymbol{\ell}) \geq 2q N$. Moreover, since $(\boldsymbol{k}, \boldsymbol{\ell})$ satisfies the zero momentum condition 
 $\boldsymbol{k}_1+ \cdots + \boldsymbol{k}_q = \boldsymbol{\ell}_1+ \cdots + \boldsymbol{\ell}_q$, we have $(2q-1)\mu_2 (\boldsymbol{k}, \boldsymbol{\ell}) \geq \mu_1(\boldsymbol{k}, \boldsymbol{\ell}) $ and so, finally we deduce that $\mu_2 (\boldsymbol{k}, \boldsymbol{\ell}) > N$.
\end{proof}

Now, we decompose $L$ in two parts $L= L^{(low)}+L^{(high)}$, where $L^{(low)},L^{(high)}\in \mathscr{H}_{2q}$ are two homogeneous $\nu$-resonant polynomials of degree $2q$ defined by
$$
L^{(low)}_{\boldsymbol{k},\boldsymbol{\ell}} = \left\{ \begin{array}{lll} L_{\boldsymbol{k},\boldsymbol{\ell}} & \mathrm{if} & \mu_2(\boldsymbol{k},\boldsymbol{\ell}) < N \\  0 & \mathrm{else} \end{array} \right. \quad \mathrm{and} \quad L^{(high)}_{\boldsymbol{k},\boldsymbol{\ell}} = \left\{ \begin{array}{lll} 0 & \mathrm{if} & \mu_2(\boldsymbol{k},\boldsymbol{\ell}) < N \\  L_{\boldsymbol{k},\boldsymbol{\ell}} & \mathrm{else} \end{array} \right. .
$$
As a consequence of Lemma \ref{lem:outza}, we prove in the following lemma that $L^{(low)}$ commutes with $\mathcal{N}_{N,s}$.
\begin{lemma}\label{low} The Hamiltonians $L^{(low)}$ and $\mathcal{N}_{N,s}$ commute (i.e. $\{L^{(low)},\mathcal{N}_{N,s}\}= 0$).
\end{lemma}
\begin{proof}
First, we note that as a consequence of Proposition \ref{prop:Poisson_vs_Z}, for all $u\in H^s(\mathbb{T}^d)$, we have 
\begin{equation}
\label{eq:M2}
\{L^{(low)},\mathcal{N}_{N,s}\}(u) = -2i \! \! \! \sum_{\substack{\boldsymbol{k},\boldsymbol{\ell} \in (\mathbb{Z}^d)^q\\  \mu_2(\boldsymbol{k},\boldsymbol{\ell}) < N  }} \! \! \! ( g_{\boldsymbol{k}_1}+ \cdots+g_{\boldsymbol{k}_q} - g_{\boldsymbol{\ell}_1} - \cdots - g_{\boldsymbol{\ell}_q}   ) L_{\boldsymbol{k},\boldsymbol{\ell} } u_{\boldsymbol{k}_1} \dots u_{\boldsymbol{k}_q} \overline{u_{\boldsymbol{\ell}_1}} \dots \overline{u_{\boldsymbol{\ell}_q}}
\end{equation}
where $g_k =|k|^{2s}$ if $k\geq N$ and $g_k = (2^n)^{2s}$ if $k\in B_n$ with $n<n_{\max}$. Moreover, since $L$ is $\nu$-resonant, as a consequence of Lemma \ref{lem:outza}, if $L_{\boldsymbol{k},\boldsymbol{\ell} }\neq 0$ and $\mu_2(\boldsymbol{k},\boldsymbol{\ell}) < N $ then $(\boldsymbol{k}, \boldsymbol{\ell})$ does not satisfy  \eqref{eq:monomials_to_remove}. In other words, the sum in \eqref{eq:M2} can be restricted to the indices such that $\mu_2(\boldsymbol{k},\boldsymbol{\ell}) < N $ and $(\boldsymbol{k}, \boldsymbol{\ell})$ does not satisfy  \eqref{eq:monomials_to_remove}. Therefore, since $N$ is of the form $N=2^{n_{\max}}$, all the indices also satisfy $\mu_1(\boldsymbol{k},\boldsymbol{\ell}) < N$ and so 
$$
g_{\boldsymbol{k}_1}+ \cdots+g_{\boldsymbol{k}_q} - g_{\boldsymbol{\ell}_1} - \cdots - g_{\boldsymbol{\ell}_q} = (2^{\boldsymbol{m}_1})^{2s}+ \cdots+ (2^{\boldsymbol{m}_q})^{2s} -(2^{\boldsymbol{n}_1})^{2s}- \cdots - (2^{\boldsymbol{n}_q})^{2s}
$$
where $\boldsymbol{m}_j$ (resp. $\boldsymbol{n}_j$) is the index such that $\boldsymbol{k}_j \in B_{\boldsymbol{m}_j}$ (resp. $\boldsymbol{\ell}_j \in B_{\boldsymbol{n}_j}$). But since here we only consider indices such that $(\boldsymbol{k}, \boldsymbol{\ell})$ does not satisfy  \eqref{eq:monomials_to_remove}, $\boldsymbol{m}$ and $\boldsymbol{n}$ are equal up to a permutation (i.e. $\boldsymbol{m} \in \mathfrak{S}_q \boldsymbol{n}$) and so this sum is trivial: $g_{\boldsymbol{k}_1}+ \cdots+g_{\boldsymbol{k}_q} - g_{\boldsymbol{\ell}_1} - \cdots - g_{\boldsymbol{\ell}_q} = 0$. Therefore, as a consequence of \eqref{eq:M2}, we have proven that $\{L^{(low)},\mathcal{N}_{N,s}\}(u) =0$.
\end{proof}

We are now in position to prove Proposition \ref{prop:maintech}.
\begin{proof}[Proof of Proposition \ref{prop:maintech}] In view of Lemma \ref{low} it remains to prove, under the hypothesis of the proposition, that 
\begin{equation}\label{est-high}
|\{ \mathcal{N}_{N,s}, L^{(high)} \}(u)| \lesssim_s q^{2s+1}  N^{-\eta} \| L \|_{\ell^\infty}\| u\|_{\ell^1_\eta} \| u\|_{\ell^1}^{2q-3} \| u\|_{H^s}^{2}.
\end{equation}
Following the notations introduced for \eqref{eq:M2} we get
\begin{align*}
\big|\{L^{(high)},\mathcal{N}_{N,s}\}(u)\big| &= 2 \big| \! \! \! \sum_{\substack{\boldsymbol{k},\boldsymbol{\ell} \in (\mathbb{Z}^d)^q\\  \mu_2(\boldsymbol{k},\boldsymbol{\ell}) \geq N  }} \! \! \! ( g_{\boldsymbol{k}_1}+ \cdots+g_{\boldsymbol{k}_q} - g_{\boldsymbol{\ell}_1} - \cdots - g_{\boldsymbol{\ell}_q}   ) L_{\boldsymbol{k},\boldsymbol{\ell} } u_{\boldsymbol{k}_1} \dots u_{\boldsymbol{k}_q} \overline{u_{\boldsymbol{\ell}_1}} \dots \overline{u_{\boldsymbol{\ell}_q}}\big|\\
&\leq 2 \| L \|_{\ell^\infty} \! \! \! \! \! \! \! \! \! \! \! \! \sum_{\substack{\boldsymbol{k},\boldsymbol{\ell} \in (\mathbb{Z}^d)^q\\  \mu_2(\boldsymbol{k},\boldsymbol{\ell}) \geq N, \ \Omega(\boldsymbol{k},\boldsymbol{\ell})\leq \nu  \\ \boldsymbol{k}_1+ \cdots + \boldsymbol{k}_q = \boldsymbol{\ell}_1+ \cdots + \boldsymbol{\ell}_q}} \! \! \! \! \! \! \! \! \! \! \! \! \big| g_{\boldsymbol{k}_1}+ \cdots+g_{\boldsymbol{k}_q} - g_{\boldsymbol{\ell}_1} - \cdots - g_{\boldsymbol{\ell}_q}   \big| |u_{\boldsymbol{k}_1} \dots u_{\boldsymbol{k}_q} \overline{u_{\boldsymbol{\ell}_1}} \dots \overline{u_{\boldsymbol{\ell}_q}}|.
\end{align*}
First we order the first two indices of   $(\boldsymbol{k},\boldsymbol{\ell})$ in such a way that $\mu_1(\boldsymbol{k},\boldsymbol{\ell})=|\boldsymbol{k}_1|$ and $\mu_2(\boldsymbol{k},\boldsymbol{\ell})=|\boldsymbol{k}_2|$ or $\mu_2(\boldsymbol{k},\boldsymbol{\ell})=|\boldsymbol{\ell}_1|$:

$$\sum_{\substack{\boldsymbol{k},\boldsymbol{\ell} \in (\mathbb{Z}^d)^q, \ \Omega(\boldsymbol{k},\boldsymbol{\ell})\leq \nu\\  \mu_2(\boldsymbol{k},\boldsymbol{\ell}) \geq N\\ \boldsymbol{k}_1+ \cdots + \boldsymbol{k}_q = \boldsymbol{\ell}_1+ \cdots + \boldsymbol{\ell}_q  }} \! \! \!\big| g_{\boldsymbol{k}_1}+ \cdots+g_{\boldsymbol{k}_q} - g_{\boldsymbol{\ell}_1} - \cdots - g_{\boldsymbol{\ell}_q}   \big| |u_{\boldsymbol{k}_1} \dots u_{\boldsymbol{k}_q} \overline{u_{\boldsymbol{\ell}_1}} \dots \overline{u_{\boldsymbol{\ell}_q}}|\leq (2q)^2\big( \Sigma_1+\Sigma_2 \big)$$
where
$$ \Sigma_1= \sum_{\substack{\boldsymbol{k},\boldsymbol{\ell} \in (\mathbb{Z}^d)^q, \ \Omega(\boldsymbol{k},\boldsymbol{\ell})\leq \nu\\ \mu_1(\boldsymbol{k},\boldsymbol{\ell})=|\boldsymbol{k}_1|\geq \mu_2(\boldsymbol{k},\boldsymbol{\ell})=|\boldsymbol{k}_2| \geq N \\ \boldsymbol{k}_1+ \cdots + \boldsymbol{k}_q = \boldsymbol{\ell}_1+ \cdots + \boldsymbol{\ell}_q  }} \! \! \! 
\big| g_{\boldsymbol{k}_1}+ \cdots+g_{\boldsymbol{k}_q} - g_{\boldsymbol{\ell}_1} - \cdots - g_{\boldsymbol{\ell}_q}   \big| |u_{\boldsymbol{k}_1} \dots u_{\boldsymbol{k}_q} \overline{u_{\boldsymbol{\ell}_1}} \dots \overline{u_{\boldsymbol{\ell}_q}}|$$
and
$$\Sigma_2=\sum_{\substack{\boldsymbol{k},\boldsymbol{\ell} \in (\mathbb{Z}^d)^q\\   \mu_1(\boldsymbol{k},\boldsymbol{\ell})=|\boldsymbol{k}_1|\geq \mu_2(\boldsymbol{k},\boldsymbol{\ell})=|\boldsymbol{\ell}_1| \geq N \\ \boldsymbol{k}_1+ \cdots + \boldsymbol{k}_q = \boldsymbol{\ell}_1+ \cdots + \boldsymbol{\ell}_q   }} \! \! \!
\big| g_{\boldsymbol{k}_1}+ \cdots+g_{\boldsymbol{k}_q} - g_{\boldsymbol{\ell}_1} - \cdots - g_{\boldsymbol{\ell}_q}   \big| |u_{\boldsymbol{k}_1} \dots u_{\boldsymbol{k}_q} \overline{u_{\boldsymbol{\ell}_1}} \dots \overline{u_{\boldsymbol{\ell}_q}}| .$$
We begin by estimating $\Sigma_1$. Using $\Omega(\boldsymbol{k},\boldsymbol{\ell})\leq \nu\leq1$, we get 
$$
|\boldsymbol{k}_1|^2+|\boldsymbol{k}_2|^2\leq (2q-2)|\mu_3(\boldsymbol{k},\boldsymbol{\ell})|^2 +2q (2\pi)^{-d/2} \| V\|_{\ell^\infty} \leq 2q (\| V\|_{\ell^\infty}+1) \langle\mu_3(\boldsymbol{k},\boldsymbol{\ell})\rangle^2.
$$
  Hence, since $ g_{k}\leq|k|^{2s}$ for  any integer $k$, 
$$
\big| g_{\boldsymbol{k}_1}+ \cdots+g_{\boldsymbol{k}_q} - g_{\boldsymbol{\ell}_1} - \cdots - g_{\boldsymbol{\ell}_q}   \big|\leq (2(2q)^s(\| V\|_{\ell^\infty}+1)^s+q-2) \langle\mu_3(\boldsymbol{k},\boldsymbol{\ell})\rangle^{2s} \lesssim_s q^{s+1} \langle \mu_3(\boldsymbol{k},\boldsymbol{\ell})\rangle^{2s}.
$$
and thus, setting $v^{(j)}_h = |u_h|$ if $h\leq q$ and  $v^{(j)}_h = |u_{-h}|$ else, by Young we have
\begin{align*} \Sigma_1&\lesssim_s q^{s+1}  \! \! \! \! \! \! \! \! \! \! \! \!  \! \! \! \! \! \!  \sum_{\substack{\boldsymbol{k},\boldsymbol{\ell} \in (\mathbb{Z}^d)^q\\\mu_1(\boldsymbol{k},\boldsymbol{\ell})=|\boldsymbol{k}_1|\geq \mu_2(\boldsymbol{k},\boldsymbol{\ell})=|\boldsymbol{k}_2| \geq N \\ \boldsymbol{k}_1+ \cdots + \boldsymbol{k}_q = \boldsymbol{\ell}_1+ \cdots + \boldsymbol{\ell}_q  }} \! \! \!  \! \! \! \! \! \! \! \! \! \!  \langle\mu_3(\boldsymbol{k},\boldsymbol{\ell})\rangle^{2s} |u_{\boldsymbol{k}_1} \dots u_{\boldsymbol{k}_q} \overline{u_{\boldsymbol{\ell}_1}} \dots \overline{u_{\boldsymbol{\ell}_q}}|
= q^{s+1} \! \! \! \! \! \! \! \! \! \! \! \! \! \! \! \sum_{ \substack{ \boldsymbol{h} \in  (\mathbb{Z}^d)^{2q}  \\ \mu_1(\boldsymbol{h})=|\boldsymbol{h}_1|\geq \mu_2(\boldsymbol{h})=|\boldsymbol{h}_2| \geq N \\ \boldsymbol{h}_1 + \cdots+ \boldsymbol{h}_{2q} =0 } } \! \! \! \! \! \! \! \! \! \! \! \! \! \! \! \langle\mu_3( \boldsymbol{h}  )\rangle^{2s}  v^{(1)}_{\boldsymbol{h}_1} \cdots v^{(2q)}_{\boldsymbol{h}_{2q}} \\
&\leq q^{s+1} N^{-\eta}\sum_{j=3}^{2q} \! \! \! \! \! \! \! \! \! \! \! \! \sum_{ \substack{ \boldsymbol{h} \in  (\mathbb{Z}^d)^{2q}  \\ \mu_1(\boldsymbol{h})=|\boldsymbol{h}_1|\geq \mu_2(\boldsymbol{h})=|\boldsymbol{h}_2| \geq N \\ \mu_3(\boldsymbol{h})=|\boldsymbol{h}_j| \\ \boldsymbol{h}_1 + \cdots+ \boldsymbol{h}_{2q} =0 } } \! \! \! \! \! \! \! \! \! \! \! \!|\boldsymbol{h}_1|^\eta |\boldsymbol{h}_2|^s  |\boldsymbol{h}_j|^s  v^{(1)}_{\boldsymbol{h}_1} \cdots v^{(2q)}_{\boldsymbol{h}_{2q}}
\leq 2q^{s+2} N^{-\eta} \|u\|_{\ell^1_\eta} \| u\|_{\ell^1}^{2q-3}\|u\|_{H^s}^2.
\end{align*}
Now we estimate $\Sigma_2$. We note that, in the sum $\Sigma_2$, 
\begin{align*}\big| g_{\boldsymbol{k}_1}+ \cdots+g_{\boldsymbol{k}_q} - g_{\boldsymbol{\ell}_1} - \cdots - g_{\boldsymbol{\ell}_q}   \big| &\leq |\boldsymbol k_1|^{2s}-|\boldsymbol\ell_1|^{2s}+(2q-2) \mu_3(\boldsymbol{k},\boldsymbol{\ell})^{2s}.
\end{align*}
On the other hand,  by the mean value theorem, $|\boldsymbol k_1|^{2s}-|\boldsymbol\ell_1|^{2s}\leq 2s|\boldsymbol k_1-\boldsymbol\ell_1||\boldsymbol k_1|^{2s-1}$ and, using the zero momentum condition we have, $|\boldsymbol k_1|\leq 2q |\boldsymbol\ell_1|$ and $|\boldsymbol k_1-\boldsymbol\ell_1|\leq 2q  \mu_3(\boldsymbol{k},\boldsymbol{\ell})$. Therefore, since $0< \eta \leq 1$, we get
$$\big| g_{\boldsymbol{k}_1}+ \cdots+g_{\boldsymbol{k}_q} - g_{\boldsymbol{\ell}_1} - \cdots - g_{\boldsymbol{\ell}_q}   \big|\leq 4s (2q)^{2s}|\boldsymbol k_1|^{s}|\boldsymbol\ell_1|^{s-1} \mu_3(\boldsymbol{k},\boldsymbol{\ell})
\lesssim q^{2s}|\boldsymbol k_1|^{s}|\boldsymbol\ell_1|^{s-\eta} \mu_3(\boldsymbol{k},\boldsymbol{\ell})^\eta$$
Thus, recalling that in $\Sigma_2$ we have $|\boldsymbol\ell_1|\geq N$, as previously we get
$$ \Sigma_2 \lesssim_s q^{2s} N^{-\eta} \!   \! \! \! \sum_{\substack{2\leq j \leq 2q \\ j\neq q+1}} \sum_{\substack{\boldsymbol{h} \in (\mathbb{Z}^d)^{2q} \\ \boldsymbol{h}_1+ \cdots + \boldsymbol{h}_{2q} = 0  }} \! \! \!  \! \! \!   \langle \boldsymbol{h}_1\rangle^s  \langle \boldsymbol{h}_{q+1}\rangle^s \langle \boldsymbol{h}_{j}\rangle^\eta  v^{(1)}_{\boldsymbol{h}_1} \cdots v^{(2q)}_{\boldsymbol{h}_{2q}}  
\leq 2q^{2s+1} N^{-\eta}\|u\|_{\ell^1_\eta} \| u\|_{\ell^1}^{2q-3}\|u\|_{H^s}^2.
$$
\end{proof}

\subsection{Proof of Theorem \ref{th-main}} 

\subsubsection{Approximation by smooth solutions} In order to justify the formal computation, we are going to prove Theorem \ref{th-main} when $u^{(0)}$ is smooth. So first, let us check that this assumption can be done without loss of generality. More precisely, we assume that Theorem \ref{th-main} holds if we add the assumption that $u^{(0)} \in C^\infty(\mathbb{T}^d)$ and we aim at proving that this assumption can be removed. 

Let $u^{(0)} \in H^{s_\star}$, where $s_\star = \max(s,s_0)$, be such that $\varepsilon = \|u^{(0)}\|_{H^{s_0}} \leq \varepsilon_0 $. Let $u^{(0,n)} \in C^\infty$, $n\geq 1$, be a sequence of functions such that 
$$
\sup_{n \geq 1}  \| u^{(0,n)} \|_{H^{s_0}} \leq \| u^{(0)} \|_{H^{s_0}} \quad \mathrm{and} \quad u^{(0,n)} \mathop{\longrightarrow}_{n\to \infty} u^{(0)} \quad \mathrm{in} \quad H^{s_\star}.
$$
Since $\varepsilon \mapsto T_\varepsilon$ is increasing (provided that $\varepsilon$ is small enough), for all $n \geq 1$, the solution $u^{(n)}$ of \eqref{NLS} with initial condition $u^{(0,n)}$ satisfies $u^{(n)}\in C^\infty([-T_\varepsilon,T_\varepsilon] \times  \mathbb{T}^d   ; \mathbb{C})$ and 
$$
M:= \sup_{n\geq 1}\sup_{ |t| \leq T_\varepsilon} \| u^{(n)} \|_{H^{s_\star}} < \infty. 
$$
We are going to prove that $u^{(n)}$ is of Cauchy in $C^0(  [-T_\varepsilon,T_\varepsilon] ;  H^{s_\star}(\mathbb{T}^d)  )$. Indeed, by Duhamel, it satisfies
\begin{equation}
\label{eq:Dudu}
u^{(n)}(t) = e^{i t (\Delta - V\ast)} u^{(n,0)} + \int_{0}^t e^{i (t-\tau) (\Delta - V\ast)}  |u(\tau)|^{2p}u(\tau)  \mathrm{d}\tau
\end{equation}
and so, since $H^{s_\star}$ is an algebra (because $s_\star \geq s_0 >d/2$), we have
$$
\| u^{(n)}(t) - u^{(m)}(t) \|_{H^{s_\star}} \leq \| u^{(0,n)} - u^{(0,m)} \|_{H^{s_\star}} + p C_{s_\star} M^{2p}   \int_{[0;t]} \| u^{(n)}(\tau) - u^{(m)}(\tau) \|_{H^{s_\star}} \mathrm{d}\tau
$$
where $C_{s_\star}>0 $ is a constant depending only on $s_\star$. Therefore as a consequence of Gr\"onwall's inequality, we have
$$
\sup_{|t|\leq T_\varepsilon} \| u^{(n)}(t) - u^{(m)}(t) \|_{H^{s_\star}} \leq \| u^{(0,n)} - u^{(0,m)} \|_{H^{s_\star}}  e^{ p C_{s_\star} M^{2p} T_\varepsilon }
$$
which proves that $u^{(n)}$ is of Cauchy in $C^0(  [-T_\varepsilon,T_\varepsilon] ;  H^{s_\star}(\mathbb{T}^d)  )$. This space being a Banach space, we denote by $u \in C^0(  [-T_\varepsilon,T_\varepsilon] ;  H^{s_\star}(\mathbb{T}^d)  )$ its limit. Passing to the limit in \eqref{eq:Dudu}, we deduce that $u \in C^1(  [-T_\varepsilon,T_\varepsilon] ;  H^{s_\star-2}(\mathbb{T}^d)  )$ is also a solution of \eqref{NLS} on $[-T_\varepsilon, T_\varepsilon]$. Moreover, since $s_\star \geq s$, $\|u^{(n)}\|_{L^\infty H^s}$ goes to $\|u\|_{L^\infty H^s}$ as $n$ goes to $+\infty$, which proves that $u$ also satisfies the bound $\|u\|_{L^\infty H^s} \lesssim_s \| u^{(0)} \|_{H^s}$.

\subsubsection{Setting of the bootstrap} Now we focus more directly on the proof of Theorem \ref{th-main}. We assume that $V \in \mathcal{V}$ is fixed (the set $\mathcal{V}$ being defined in Definition \ref{def:setV}). Thanks to the previous step, from now we assume without loss of generality that $u^{(0)} \in C^\infty(\mathbb{T}^d ; \mathbb{C})$ satisfies $\varepsilon := \| u^{(0)} \|_{H^{s_0}} \leq \varepsilon_0$ where $\varepsilon_0>0$ is a constant depending only\footnote{ and also on $V$ and $d$ but we do not track these dependencies.} on $s_0>d/2$  which will be determined later (see formula \eqref{eq:defeps0} below).

We denote by $u\in C^0((-T_-,T_+);H^{s_0}) \cap C^1((-T_-,T_+);H^{s_0-2})$ the maximal solution of \eqref{NLS} associated with $u^{(0)}$, i.e. $T_+>0$ satisfies 
$$
T_+ = +\infty \quad \mathrm{or} \quad \limsup_{t\to +\infty} \| u \|_{H^{s_0}} = +\infty.
 $$
 Of course $T_-$ is defined similarly. Since by assumption $u^{(0)} \in C^\infty$, for all $s\geq 0$, $u^{(0)} \in H^s$, and thus, since the non-linearity enjoys tame estimates, $u\in C^0((-T_-,T_+);H^{s})$ for all $s\geq 0$. Therefore, since $C^\infty(\mathbb{T}^d) = H^\infty(\mathbb{T}^d)$, it is clear that  
 $$
 u\in C^\infty((-T_-,T_+) \times  \mathbb{T}^d   ; \mathbb{C}) \subset C^1((-T_-,T_+);\ell^1).
 $$
 From now, without loss of generality, we only consider non-negative times. We consider a constant $G_{s_0}>1$ depending only on $s_0$ and that will be determined later (see formula \eqref{eq:defGs0} below). In order to prove that $T_+ > T_\varepsilon$ and that $\| u(t) \|_{H^{s_0}} \leq G_{s_0} \| u^{(0)} \|_{H^{s_0}}$ for all $t\in [0,T_\varepsilon]$, by a standard bootstrap argument, it is enough to prove that 
 \begin{equation}
 \label{eq:the_bootstrap}
  \left. \begin{array}{lll} 0\leq T < \min(T_\varepsilon,T_+) \\ 
  \forall 0\leq t \leq T, \quad \| u(t) \|_{H^{s_0}} \leq G_{s_0} \| u^{(0)} \|_{H^{s_0}} \end{array} \right\} \quad \Rightarrow \quad \| u(T) \|_{H^{s_0}} < G_{s_0} \| u^{(0)} \|_{H^{s_0}}.
 \end{equation}
The estimate of the $H^s$ norms for $s\neq s_0$ will just be a byproduct of the proof (see estimate \eqref{eq:growthHsfinal} below).
 
 \subsubsection{Parameters and change of variable} Following \eqref{eq:the_bootstrap}, from now and until the end of this proof, we consider $T>0$ such that $T < \min(T_\varepsilon,T_+)$ and for all  $t \in  [0, T],  \| u(t) \|_{H^{s_0}} \leq G_{s_0} \| u^{(0)} \|_{H^{s_0}} $.\\ 
We consider the following parameters which will be optimized later with respect to $\varepsilon\equiv \|u^{(0)} \|_{H^{s_0}}$ (see formula \eqref{eq:defN} and \eqref{eq:defr} below) :
\begin{itemize}
\item $N \geq 1$ is integer of the form $N=2^{n_{\max}}$ with $n_{\max} \in \mathbb{N}$,
\item $r\geq 2$ is an integer (it will be the order of the Birkhoff normal form),
\item $\nu>0$ is the size of the small divisors in the Birkhoff normal form. In order to apply Proposition \ref{prop:maintech} and to have small divisors as large as possible, we set
\begin{equation}
\label{eq:defnu}
\nu := \widetilde{\gamma} r^{-4} \big(\log_2 \big( 2r N \big) \big)^{-(2r+1)} 
\end{equation}
where $\widetilde{\gamma} = \min(\gamma,1)$ and $\gamma>0$ is the constant associated with the non-resonance of $V$ (see Definition \ref{def:setV}). We note that by construction we have $\nu<1$.
\end{itemize} 
Since $s_0>d/2$, we set
$$
K_{s_0} := \sqrt{\sum_{k\in \mathbb{Z}^d} \langle k \rangle^{-2s_0}}
$$
and by Cauchy-Schwarz we have $\| \cdot \|_{\ell_1} \leq K_{s_0}\| \cdot \|_{H^{s_0}}$.

\medskip

We recall that $u$ satisfies
 $$
\forall t \in [0,T], \quad i \partial_t u(t) =  \nabla \frac{H}2(u(t))
 $$
where $H$, the Hamiltonian of \eqref{NLS}, is given by \eqref{eq:HamiltofNLS}. Therefore, we apply the Birkoff normal form Theorem \ref{thm:Birk} to the Hamiltonian $H$. In order to apply the changes of variables to $u(t)$, the parameters we are going to design will satisfy the constraint 
\begin{equation}
\label{eq:CFL1}
G_{s_0} K_{s_0} \varepsilon   < \frac{\sqrt{\nu}}{Cr} = \rho.
\end{equation}
Therefore, we have
$$
\forall t \in [0,T], \quad \| u(t) \|_{\ell^1} < \rho 
$$
and so it makes sense to consider
$$
v(t) := \tau^{(0)}(u(t)).
$$
Note that, as a consequence of Theorem \ref{thm:Birk}, we have $\| v(t) \|_{\ell^1} <2\rho$ and 
$$
u(t) = \tau^{(1)}(v(t)).
$$
Moreover, thanks to \eqref{eq:inflationHs}, we have
$$
\forall t\in [0,T],\forall s\geq 0, \quad M_s^{-1} \| v(t) \|_{H^s} \leq  \| u(t) \|_{H^s} \leq M_s \| v(t) \|_{H^s} 
$$
where $M_s\geq 1$ is a constant depending only on $s$.
Finally, we aim at proving that
\begin{equation}
\label{eq:newynamic}
i\partial_t v(t) = \frac12 \nabla (H\circ \tau^{(1)}) (v(t)).
\end{equation}
Recalling that $u\in C^1((-T_-,T_+);\ell^1)$ and $\tau^{(0)}$ is smooth in $\ell^1$, by composition $v\in C^1([0,T];\ell^1)$ and we have
$$
i\partial_t v(t) = i\partial_t  \tau^{(0)}(u(t)) = i \mathrm{d} \tau^{(0)} (u(t)) ( \partial_t u(t)) = -\frac{i}2 \mathrm{d} \tau^{(0)} (u(t)) (  i (\nabla H) \circ \tau^{(1)} (v(t)))   .
$$
Therefore, to get \eqref{eq:newynamic}, we only have to prove that
\begin{equation}
\label{eq:pouettepouette}
 \mathrm{d} \tau^{(0)} (u(t)) i = i [ \mathrm{d} \tau^{(1)} (v(t))  ]^* \quad \mathrm{on} \quad \ell^1
\end{equation}
where $[ \mathrm{d} \tau^{(1)} (v(t))  ]^*$ denotes the adjoint of $ \mathrm{d} \tau^{(1)} (v(t))$. 
On the one hand, since $\tau^{(1)}$ is symplectic, we note that we have
\begin{equation}
\label{eq:pouettepouette1}
[ \mathrm{d} \tau^{(1)} (v(t))  ]^* i  \mathrm{d} \tau^{(1)} (v(t))  = i
\end{equation}
and on the other hand, since $\tau^{(1)} \circ \tau^{(0)} = \mathrm{id}_{\ell^1} $ on $B_{\ell^1}(0,\rho)$, we note that
$$
\mathrm{d} \tau^{(1)} (v(t)) \mathrm{d} \tau^{(0)} (u(t)) =  \mathrm{id}_{\ell^1}.
$$
Therefore, multiplying on the right \eqref{eq:pouettepouette1} by $\mathrm{d} \tau^{(0)} (u(t))$, we get \eqref{eq:pouettepouette}.
 \subsubsection{Sobolev norm estimates} Let $s\geq 0$. We recall that the observable $\mathcal{N}_{N,s}$ is given by Definition \ref{def:NNs} and that it is equivalent to $\| \cdot \|_{H^s}^2$. Since $\tau^{(1)} : B_{\ell^1}(0,2\rho) \cap H^s \to \ell^1 \cap H^s$ is smooth, by composition $v\in C^1([0,T];H^s)$. As a consequence, by composition, we have
 $$
 \partial_t \mathcal{N}_{N,s}(v(t)) = (\nabla \mathcal{N}_{N,s}(v(t)), \partial_t v(t))_{L^2} =\frac12 \{ \mathcal{N}_{N,s} , H\circ \tau^{(1)} \}(v(t)).
 $$
Thanks to the decomposition \eqref{eq:NF} of $H\circ \tau^{(1)}$ it comes
$$
\partial_t \mathcal{N}_{N,s}(v(t)) = \sum_{q=2}^r \{\mathcal{N}_{N,s}  , L^{(2q)} \}(v(t)) + (i \nabla \mathcal{N}_{N,s} (v(t)), \nabla R(v(t)) )_{L^2}.
$$
On the one hand, we have
\begin{equation*}
\begin{split}
|(i \nabla \mathcal{N}_{N,s} (v(t)), \nabla R(v(t)) )_{L^2} | &\leq \sqrt{\mathcal{N}_{N,s} (v(t))} \|   \nabla  R(v(t)) \|_{H^s} \\
&\lesssim_s \sqrt{\mathcal{N}_{N,s} (v(t))} C^{2r} \Big(\frac{r^3}{\nu} \Big)^{r-1}  \| v(t) \|_{\ell^1}  ^{2r}  \| v(t)\|_{H^s} \\
&\lesssim_s   C^{2r} \Big(\frac{r^3}{\nu} \Big)^{r-1}  \| v(t) \|_{\ell^1}  ^{2r} \,  \mathcal{N}_{N,s} (v(t)) \\
&\lesssim_s (C M_{s_0} G_{s_0} K_{s_0})^{2r}  \Big(\frac{r^3}{\nu} \Big)^{r-1}  \varepsilon^{2r}  \mathcal{N}_{N,s} (v(t)).
\end{split}
\end{equation*}
While, on the other hand, by Proposition \ref{prop:maintech}, for all $\eta \in (0,1]$, we have\footnote{Here we used $q^{2s+1}\lesssim_s 2^q$ for all $q\in\mathbb N$.}
$$
| \{\mathcal{N}_{N,s}  , L^{(2q)} \}(v(t)) | \lesssim_s 2^q  N^{-\eta} \| L^{(2q)} \|_{\ell^\infty}\| v(t) \|_{\ell^1_\eta} \| v(t) \|_{\ell^1}^{2q-3} \| v(t) \|_{H^s}^{2}.
$$
We choose
$$
\eta \equiv \eta_{s_0} = \min \Big[ 1, \frac12 \big( s_0 - \frac{d}2 \big) \Big]
$$
in such a way that $d/2 < d/2+ \eta_{s_0}  < s_0$ and so
$$
\| v(t) \|_{\ell^1_{\eta_{s_0} }} \leq K_{d/2+ \eta_{s_0} } \| v(t) \|_{H^{s_0}}.
$$
Therefore, we have
\begin{equation*}
\begin{split}
| \{\mathcal{N}_{N,s}  , L^{(2q)} \}(v(t)) | &\lesssim_s 2^q  K_{d/2+ \eta_{s_0} }^{2q-2}  N^{-\eta} \| L^{(2q)} \|_{\ell^\infty} \| v(t) \|_{H^{s_0}}^{2q-2} \| v(t) \|_{H^s}^{2} \\
&\lesssim_s 2^q (M_{s_0} G_{s_0} K_{d/2+ \eta_{s_0} })^{2q-2}  N^{-\eta} \| L^{(2q)} \|_{\ell^\infty} \varepsilon^{2q-2}  \mathcal{N}_{N,s} (v(t)) \\
&\lesssim_s  (2 C M_{s_0} G_{s_0} K_{d/2+ \eta_{s_0} })^{2q-2}  N^{-\eta}  \Big(\frac{q^2}{\nu} \Big)^{q-2} \varepsilon^{2q-2}  \mathcal{N}_{N,s} (v(t)).
\end{split}
\end{equation*}
The parameters we are going to design will satisfy the constraint
\begin{equation}
\label{eq:CFL2}
4 C M_{s_0} G_{s_0} K_{d/2+ \eta_{s_0} } \varepsilon   < \frac{\sqrt{\nu}}{Cr}.
\end{equation}
Thus we get
$$
| \{\mathcal{N}_{N,s}  , L^{(2q)} \}(v(t)) | \lesssim_s  2^{-2q}  N^{-\eta}   \mathcal{N}_{N,s} (v(t))
$$
and so
$$
| \partial_t \mathcal{N}_{N,s}(v(t)) | \lesssim_s [ N^{-\eta}+  (C M_{s_0} G_{s_0} K_{s_0})^{2r}  \Big(\frac{r^3}{\nu} \Big)^{r-1}  \varepsilon^{2r} ] \mathcal{N}_{N,s} (v(t)).
$$
Therefore, to homogenize this sum, we fix the parameter $N$ in such a way that
\begin{equation}
\label{eq:defN}
 2^{-1} \varepsilon^{-\frac{r}{\eta}} \leq N <  \varepsilon^{-\frac{r}{\eta}}.
\end{equation}
and so $N^{-\eta} \leq 2^{\eta} \varepsilon^r \leq 2 \varepsilon^r$. Recalling that $\nu$ is defined as 
$\nu = \widetilde{\gamma} r^{-4} \big(\log_2 \big( 2r N \big) \big)^{-(2r+1)} $ (see  \eqref{eq:defnu}), 
we have (using that $\varepsilon<1$ and $\log_2 (2r) \leq r$)
\begin{equation*}
\begin{split}
\nu ^{-(r-1)} \leq \widetilde{\gamma}^{-r} r^{4r} \big( \log_2 (2r \varepsilon^{-\frac{r}{\eta}}) \big)^{2r^2} &\leq \widetilde{\gamma}^{-r} r^{4r} \big( \log_2 (2r) - \frac{r}{\eta} \log(\varepsilon)  \big)^{2r^2} \\ 
&\leq 2^{2r^2}\widetilde{\gamma}^{-r} r^{3r^2} \eta^{-2r^2}  \log^{2r^2}(\varepsilon^{-1})  
\end{split}
\end{equation*}
and so
$$
 | \partial_t \mathcal{N}_{N,s}(v(t)) | \lesssim_s [ 1+  (C \widetilde{\gamma}^{-1} M_{s_0} G_{s_0} K_{s_0})^{2r}  (2 \eta^{-1})^{2r^2} r^{4r^2} \log^{2r^2}(\varepsilon^{-1}) \,    \varepsilon^{r} ] \, \varepsilon^r \mathcal{N}_{N,s} (v(t)).
$$
We fix $r$ as an integer satisfying
\begin{equation}
\label{eq:defr}
\frac{|\log \varepsilon |}{4 \log |\log \varepsilon |} \leq r \leq \frac{|\log \varepsilon |}{3 \log |\log \varepsilon |} =: r_\varepsilon.
\end{equation}
Note that this definition makes sense provided that $\varepsilon_0$ is smaller than an universal constant.
Therefore, we have
$$
 \log^{2r^2}(\varepsilon^{-1}) \,    \varepsilon^{r} \leq \exp \Big( - \frac1{36} \, \frac{(\log \varepsilon)^2}{ \log |\log \varepsilon |} \Big).
$$
and so, since
\begin{equation}
\label{eq:defeps0}
(C \widetilde{\gamma}^{-1} M_{s_0} G_{s_0} K_{s_0})^{2r_\varepsilon}  (2 \eta_{s_0}^{-1})^{2r_\varepsilon^2} r_\varepsilon^{4r_\varepsilon^2} \exp \Big( - \frac1{36} \, \frac{(\log \varepsilon)^2}{ \log |\log \varepsilon |} \Big) \mathop{\longrightarrow}_{\varepsilon \to 0} 0
\end{equation}
we deduce that provided that $\varepsilon_0$ is small enough with respect to a constant depending only on $s_0$, this quantity is bounded by $1$ (and \eqref{eq:CFL1}, \eqref{eq:CFL2} are satisfied), and so that we have
$$
| \partial_t \mathcal{N}_{N,s}(v(t)) | \leq \Upsilon_s T_\varepsilon^{-1} \mathcal{N}_{N,s} (v(t)) \quad \mathrm{where} \quad T_\varepsilon = \exp\Big( \frac{|\log \varepsilon |^2}{4 \log |\log \varepsilon |} \Big)
$$
and $\Upsilon_s>1$ is a constant depending only on $s$. As a consequence, by Gr\"onwall, since $0\leq t\leq T < T_\varepsilon$ we get
$$
 \mathcal{N}_{N,s}(v(t)) \leq e^{\Upsilon_s t  T_\varepsilon^{-1}}\mathcal{N}_{N,s}(v(0)) < e^{\Upsilon_s} \mathcal{N}_{N,s}(v(0))
$$
and so
\begin{equation}
\label{eq:growthHsfinal}
\forall t \in [0,T], \quad \| u(t) \|_{H^s}^2 < M_s^4 2^{2s} e^{\Upsilon_s} \| u^{(0)} \|_{H^s}^2.
\end{equation}
Therefore, to conclude the bootstrap (see \eqref{eq:the_bootstrap}) it is enough to set\footnote{note that the constants $M_s$ and $\Upsilon_s$ do not depend of the choice of $G_{s_0}$.}
\begin{equation}
\label{eq:defGs0}
G_{s_0} :=  M_{s_0}^2 2^{s_0} e^{\Upsilon_{s_0}}.
\end{equation}

\end{document}